\documentclass[a4paper]{article}

\usepackage{listings}

\lstdefinestyle{Python}{
    language        =   Python, 
    basicstyle      =   \zihao{-5}\ttfamily,
    numberstyle     =   \zihao{-5}\ttfamily,
    keywordstyle    =   \color{blue},
    keywordstyle    =   [2] \color{teal},
    stringstyle     =   \color{magenta},
    commentstyle    =   \color{red}\ttfamily,
    breaklines      =   true,   
    columns         =   fixed,  
    basewidth       =   0.5em,
}

\usepackage{amsthm}

\usepackage{appendix}
\usepackage[all]{xy}\usepackage[latin1]{inputenc}        
\usepackage[dvips]{graphics,graphicx}
\usepackage{amsfonts,amssymb,amsmath,color,mathrsfs, amstext}
\usepackage{amsbsy, amsopn, amscd, amsxtra, amsthm,authblk, enumerate}
\usepackage{upref}
\usepackage[colorlinks,
            linkcolor=red,
            anchorcolor=red,
            citecolor=red
            ]{hyperref}

\usepackage{geometry}
\usepackage{float}

\usepackage[mathlines]{lineno}

\usepackage{yhmath}

\usepackage[normalem]{ulem}

\numberwithin{equation}{section}

\newcommand{\E}{\mathbb{E}}

\newcommand{\bbR}{\mathbb{R}}

\newtheorem{theorem}{Theorem}[section]
\newtheorem{lemma}{Lemma}[section]
\newtheorem{assumption}{Assumption}[section]
\newtheorem{proposition}{Proposition}[section]
\newtheorem{remark}{Remark}[section]
\newtheorem{corollary}{Corollary}[section]

\begin{document}
\title{Geometric ergodicity of SGLD via reflection coupling}
\author[a,d]{Lei Li\thanks{E-mail:leili2010@sjtu.edu.cn}}
\author[b]{Jian-Guo Liu\thanks{E-mail:jliu@math.duke.edu}}
\author[c]{Yuliang Wang\thanks{E-mail:YuliangWang$\_$math@sjtu.edu.cn}}
\affil[a]{School of Mathematical Sciences, Institute of Natural Sciences, MOE-LSC, Qing Yuan Research Institute, Shanghai Jiao Tong University, Shanghai, 200240, P.R.China.}
\affil[b]{Department of Mathematics, Department of Physics, Duke University, Durham, NC 27708, USA.}
\affil[c]{School of Mathematical Sciences, Shanghai Jiao Tong University, Shanghai, 200240, P.R.China.}
\affil[d]{Shanghai Artificial Intelligence Laboratory}
\date{}
\maketitle
\begin{abstract}
    We consider the geometric ergodicity of the Stochastic Gradient Langevin Dynamics (SGLD) algorithm under nonconvexity settings. Via the technique of reflection coupling, we prove the Wasserstein contraction of SGLD when the target distribution is log-concave only outside some compact set. The time discretization and the minibatch in SGLD introduce several difficulties when applying the reflection coupling, which are addressed by a series of careful estimates of conditional expectations.  As a direct corollary, the SGLD with constant step size has an invariant distribution and we are able to obtain its geometric ergodicity in terms of $W_1$ distance. The generalization to non-gradient drifts is also included.
\end{abstract}

\section{Introduction}
The Stochastic Gradient Langevin Dynamics (SGLD), first introduced by Welling and Teh \cite{welling2011bayesian}, has attracted a lot of attention in various areas \cite{mou2018generalization,zou2021faster,brosse2018promises}. The SGLD algorithm and its variants have shown exceptional performance when dealing with many practical sampling or optimization tasks. As an online algorithm, SGLD can be viewed as adding independent white noise to the well-known the classical machine learning algorithm, Stochastic Gradient Descent (SGD), making it useful for sampling tasks. The goal here is to generate samples from a target distribution $\pi$. In fact, SGLD is a Markov process that approximates the overdamped Langevin diffusion whose invariant measure is the target distribution $\pi$ in the sampling task. Here the approximation is realized by using random batch to compute the drift at discrete time $T_k := k\eta$, and $\eta$ is the constant time step (or learning rate). In this paper, our primary focus is on the theoretical study of SGLD's convergence to the invariant measure and its convergence rate.

Let us first explain the SGLD method. Suppose that the distribution of interest is $\pi\propto \exp(-\beta U)$ where $U: \bbR^d\to \bbR$ is the free energy and  $\beta>0$ is a positive constant describing the inverse temperature of the system. One effective way to sample from the target $\pi$ is through the following overdamped Langevin diffusion, whose invariant measure is exactly $\pi$:
\begin{equation*}
    dX = -\nabla U(X)dt + \sqrt{2\beta^{-1}} dW, \quad X|_{t=0}=X_0,
\end{equation*}
where $W$ is the Brownian motion in $\mathbb{R}^d$. To numerically compute the sampling procedure, one often uses the Euler-Maruyama scheme. Given the time step (or learning rate) $\eta_k$ at $k$-th iteration, and denote $T_k := \sum_{i=0}^{k-1} \eta_i$, the Euler-Maruyama scheme iterates as follows:
\begin{equation*}
    \hat{X}_{T_{k+1}} = \hat{X}_{T_k} - \eta_k \nabla U(\hat{X}_{T_k}) + \sqrt{2\beta^{-1}}(W_{T_{k+1}} - W_{T_k}).
\end{equation*}

The key idea of SGLD is to reduce the computation cost by using the random batch. In fact, in various sampling and optimization tasks from machine learning and data science, people deal with the potential $U(\cdot)$ coming from high dimensional large-scaled data with size $N$. Often $U(\cdot)$ is of the form $U(\cdot)=\mathbb{E}_{\xi}\left[U^{\xi}(\cdot)\right]$, which is the expected value of a function depending on a random variable $\xi  \in \mathcal{S}$. However, usually we do not have any knowledge of the data's distribution, and the only realistic approach to computing $U(\cdot)$ is through the random batch of a fixed small size $S\ll N$ repeatedly and independently used at each $T_k$ (see \eqref{eq:sglddiscrete} for the details). When $k$ goes large such that $kS \approx N \gg 1$, the random batch approximation for $U(\cdot)=\mathbb{E}_{\xi}\left[U^{\xi}(\cdot)\right]$ is then realized accumulatively due to the law of large numbers, and meanwhile the computational cost at each step is significantly reduced 
 since $S\ll N$.
In practice, one often has $U(x)=U_0(x)+\frac{1}{N}\sum_{i=1}^N \ell_i(x)$ and, as in the stochastic gradient descent algorithm \cite{robbins1951stochastic,feng2018semi}, $\xi$ often represents the minibatch of $\{1,\cdots, N\}$ (In this case, for fixed batch-size $S$ (a determined constant), $\xi$ belongs to the set $\mathcal{S}=\{(a_1,\dots,a_S): a_i (1\leq i \leq S)$ are $S$ different random numbers uniformly chosen from $\{1,\dots, N \}$ $\}$. For $\xi = (a_1,\dots,a_S)$, the corresponding unbiased estimate $U^{\xi}$ is $U^{\xi}(x)=U_0(x)+\frac{1}{S}\sum_{i=1}^S \ell_{a_i}(x)$.)
The general form of SGLD iteration can be written in the following form.
\begin{equation}\label{eq:sglddiscrete}
    \bar{X}_{T_{k+1}} = \bar{X}_{T_k} - \eta_k \nabla U^{\xi_k}(\bar{X}_{T_k}) + \sqrt{2\beta^{-1}}(W_{T_{k+1}} - W_{T_k}).
\end{equation}
Here, $U^{\xi_k}$ is an unbiased estimate for $U$, and thus $\nabla U^{\xi_k}$ is also an unbiased estimate for $\nabla U$. As mentioned above, $\xi_k$ often represents the random mini-batch of some fixed batch size $S$ at time $t_k$, and $\{\xi_k\}_{k=0}^{\infty}$ are i.i.d.. Also,  In our analysis, we also consider the following continuous version, which is a continuous-time Markov process with continuous path:
\begin{equation}\label{eq:sgld}
    \bar{X}_t = \bar{X}_{T_k} - \int_{T_k}^t \nabla U^{\xi_k} (\bar{X}_{T_k}) ds + \int_{T_k}^t\sqrt{2\beta^{-1}} dW_s,  \quad t \in [T_k, T_{k+1}),\quad k = 0,1,\dots,
\end{equation}
and the corresponding differential form
\begin{equation}\label{eq:sglddiff}
    d\bar{X}_t = -\nabla U^{\xi_k} (\bar{X}_{T_k}) dt + \sqrt{2\beta^{-1}} dW,  \quad  \bar{X}_t|_{t = T_k} = \bar{X}_{T_k},\quad t \in [T_k, T_{k+1})\quad k = 0,1,\dots.
\end{equation}
Note that the value of \eqref{eq:sgld} at time grid $T_k$ is exactly that of \eqref{eq:sglddiscrete}, so it is enough to study the continuous version to obtain estimate for SGLD at $t = T_k$.

Recent decades have witnessed great development of theoretical research for sampling error bound of SGLD \cite{farghly2021time,mou2018generalization,zou2021faster,chau2021stochastic,zhang2023nonasymptotic,li2022sharp}. With SGLD considered a numerical scheme for the overdamped Langevin diffusion, one is naturally motivated to study the algorithm's approximation accuracy. Specifically, when comparing the densities $\bar{\rho}_t$, $\rho_t$ of time marginal distributions of SGLD and overdamped Langevin diffusion, respectively, the authors of \cite{li2022sharp} proved that $H(\bar{\rho}_t \| \rho_t) \leq C\eta^2$, where $H(\cdot \| \cdot)$ is the relative entropy (or KL-divergence) and recall that $\eta$ is the constant learning rate. Consequently, using ergodicity of the overdamped Langevin diffusion which can be derived provided that its invariant measure $\pi$ satisfies the log-Sobolev inequality, one can estimate the Wasserstein or total variation distance between $\bar{\rho}_t$ and the target $\pi$: $W_p (\bar{\rho}_t,\pi)$, $TV(\bar{\rho}_t, \pi)$ $\leq Ce^{-Ct} + C\eta^{\alpha}$ for some rate $\alpha \leq 1$ and $p = 1,2$. Notably, recently the authors of \cite{li2022sharp} obtained the optimal rate $\alpha = 1$ while in some other literature like \cite{farghly2021time,mou2018generalization,zou2021faster,chau2021stochastic,zhang2023nonasymptotic}  $\alpha$ is no larger than $\frac{1}{2}$. Moreover, under the global strongly-log-concaveness assumption for the target $\pi$, using the synchronous coupling method, it can be proved that the SGLD algorithm itself as a Markov chain has an invariant measure $\tilde{\pi}$, and $\bar{\rho}_t$ converges to $\tilde{\pi}$ exponentially in time in terms of Wasserstein-2 distance \cite{brosse2018promises}. 
However, the stringent requirement of global strong-log-concaveness potentially restricts the broader applicability of these results. For instance, this result would not give a good theoretical guarantee of convergence when one is sampling from Gaussian mixture distributions. The question of the existence and uniqueness of  $\tilde{\pi}$, as well as the algorithm's ergodicity when one only assumes strong-log-concaveness of the target distribution $\pi$ outside some compact set, remains an open area for future research. The primary objective of this paper is to resolve such problem; specifically, we aim to study the geometric ergodicity of SGLD, assuming strong-log-concaveness of the target distribution $\pi$ outside some compact set and some other regular Lipschitz conditions (see Section \ref{sec:assumption} for more details).

Now in order to study the geometric ergodicity of SGLD, we use the classical coupling method \cite{dubrushin79}, and in particular we apply the method of reflection coupling \cite{lindvall1986coupling,eberle2011reflection,eberle2016reflection}, which was originally designed to study the contraction property of many continuous SDEs. 
Here, we give a brief summary of how the reflection coupling method is adopted to study the geometric ergodicity of SGLD. Consider the two time marginal distributions $\mu_t$, $\nu_t$ of SGLD \eqref{eq:sgld}, starting from the initial distributions $\mu_0$, $\nu_0$, respectively. We aim to prove the contraction property under the Wassertein-1 distance: $W_1(\mu_t,\nu_t) \lesssim e^{-ct}W_1(\mu_0,\nu_0)$. The coupling method then reduces this goal to find some paired dynamics $(\bar{X}_t, \bar{Y}_t)$ satisfying the laws of $\bar{X}_t$, $\bar{Y}_t$ are $\mu_t$, $\nu_t$, respectively, and the Lyapunov exponent
$$
\gamma := \limsup_{t \rightarrow \infty} \frac{1}{t}\log \mathbb{E}|\bar{X}_t - \bar{Y}_t| \leq -c
$$
is negative for this paired dynamics ($\bar{X}_t$, $\bar{Y}_t$). Note that the geometric ergodicity arises from strong convexity of the potential $U(\cdot)$ outside some compact set. This strong convexity becomes strong monotonicity property for any two points $(x,y)$ far away, as in Lemma \ref{lmm:ass} below. Therefore, any such pair ($\bar{X}_t$, $\bar{Y}_t$) would attract each other if they are sufficiently far away.

Next, in order to construct such paired dynamics ($\bar{X}_t$, $\bar{Y}_t$), we use the key technique - reflection coupling equipped with a specific Lyapunov function $f(\cdot)$. This technique was originally designed by Lindvall and Rogers in 1986 and was developed by Eberle etc. to study the geometric ergodicity of many continuous dynamics. Here, the Lyapunov function $f(\cdot)$ defined in \eqref{eq:f_sgld} in our result is an increasing, concave function. Correspondingly, we consider the Kantorovich-Rubinstein distance $W_f(\cdot,\cdot)$ with cost function $f(\cdot)$ defined in \eqref{eq:Wf} below.
The reflection coupling methods begins with choosing the pair of initial points $(\bar{X}_0, \bar{Y}_0)$ such that $\mathbb{E}f(|\bar{X}_0-\bar{Y}_0|) = W_f(\mu_0,\nu_0)$.
Then we choose a realization $\bar{X}_t$ of SGLD \eqref{eq:sgld} such that the law of $X_t$ is $\mu_t$ and the law of $X_0$ is $\mu_0$. The key step in the reflection coupling method is that we construct a companion process $\bar{Y}_t$ with $\bar{Y}_0$ coupled above with $\bar{X}_0$ and satisfies: (i) $\bar{Y}_t$ shares the same random batch and Brownian motion with $\bar{X}_t$, and has an additional reflection term in its diffusion part (see \eqref{eq:reflectioncoupling} below); (ii) $\bar{Y}_t$ is also a realization of SGLD \eqref{eq:sgld} and the law of $\bar{Y}_t$ is $\nu_t$ (see Lemma \ref{lmm:verifyBM}). Then the contraction property mentioned above is reduced to estimation of the negative Lyapunov exponent for the paired dynamics ($\bar{X}_t$, $\bar{Y}_t$).
In fact, with this specially designed diffusion in the paired dynamics $(\bar{X}_t, \bar{Y}_t)$, we can actually prove the exponential decay in time of $\mathbb{E}f(\bar{X}_t-\bar{Y}_t)$ and therefore obtain the $W_f$-contraction (see Theorem \ref{thm:contraction} below): $$W_f(\mu_{t},\nu_{t}) \leq \mathbb{E}f(|\bar{X}_{t} - \bar{Y}_{t}|) \leq Ce^{-Ct} \mathbb{E}f(|\bar{X}_{0} - \bar{Y}_{0}|) = Ce^{-CT_k}W_f(\mu_0,\nu_0).$$ 
Notably, the key contribution of the reflection coupling is as follows: different from synchronous coupling method where $\bar{X}_t$, $\bar{Y}_t$ shares exactly the same Brownian motion, in the reflection coupling, the process $\bar{X}_{t} - \bar{Y}_{t}$ is still a diffusion process. In particular its diffusion is an anisotropic one, see the expression in \eqref{eq:Ztevolve}. Consequently, the existence of this diffusion leads to a $f''(\cdot)$ term after It\^o's calculus, see \eqref{eq:dEdt}. Then the contraction property can be obtained based on the following concave property of the constructed Lyapunov function in \eqref{eq:f_sgld}: 
\begin{equation*}
    f''(r) \lesssim -r,
\end{equation*}
for all $r$ in a bounded set.
After proving the contraction property, one can directly obtain the geometric ergodicity of SGLD (see Corollary \ref{coro:ergodicity} below) using the Banach's contraction mapping theorem. Moreover, our choice of the $f(\cdot)$ makes the two distances $W_f(\cdot,\cdot)$, $W_1(\cdot,\cdot)$ equivalent, enabling one to obtain the geometric ergodicity under the Wasserstein-1 distance. Further details regarding the formulation of such paired dynamics $(\bar{X}_t,\bar{Y}_t)$ and the construction of the Lyapunov function $f(\cdot)$ will be elaborated upon in Section \ref{section:omiited}.

These years, the reflection coupling has been instrumental in establishing the geometric ergodicity of various random dynamic systems including overdamped/underdamped Langevin diffusion \cite{eberle2019couplings,eberle2016reflection,majka2020nonasymptotic}, Hamiltonian Monte Carlo \cite{bou2020coupling,bou2021two}, first-order interacting particle systems \cite{eberle2019quantitative,durmus2020elementary}, etc.   Recently, in \cite{majka2020nonasymptotic}, the authors constructed a reflection coupling for the discrete Euler-Maruyama scheme directly and obtained the contraction and ergodicity in Wasserstein-1 and Wasserstein-2 distances without random batch. Moreover, their method also gives some estimate for the long-time behavior of SGLD, but there is an $O(\eta)$ remainder in the control coming from the variance of the random batch (see \cite{majka2020nonasymptotic}, Theorem 2.16), so intuitively the ergodicity of SGLD could not be proved directly through this estimate.  In \cite{jin2022ergodicity}, the authors studied the ergodicity of the time-continuous random batch dynamics for the interacting particle systems, using a variant of the reflection coupling. The model studied resembles SGLD but we remark that the proof there makes use of the external confining potential and regards the random batch version of the interaction as perturbation. In our setting below, we only assume the confining property of {\it the expected drift} with no external potential to help, and we will consider the freezing drift dynamics instead and show that the distance between laws of two SGLD copies will vanish to zero in Wasserstein-1 distance exponentially in time.

However, due to existence of the time discretization and the random mini-batch, there are several difficulties arising when applying the reflection coupling to analyze the SGLD algorithm, detailed as follows.
The first difficulty arises from the numerical discretization. In the time continuous interpolation \eqref{eq:sgld}, the drifts are evaluated at $T_k$ but the dynamics is evolving and the hitting time (defined in \eqref{eq:deftau}) in the reflection coupling could be between $[T_k, T_{k+1})$. This dismatch brings extra difficulty compared with the reflection coupling for time-continuous processes. 
Furthermore, when dealing with this difficulty, one needs to conduct careful estimate for the tail behavior of the multiplicative noise $\zeta_t$ (see \eqref{eq:Ztevolve} below).
In fact, although the diffusion part $\hat{W}_t$ in \eqref{eq:reflectioncoupling} or Lemma \ref{lmm:verifyBM} below is a Brownian motion, it correlates with the original Brownian motion $W_t$. Therefore, in \eqref{eq:Ztevolve} below, $\zeta_t = \int_{T_k}^t d\hat{W}_s - dW_s$ is a multiplicative noise. Estimate for this multiplicative noise is not trivial and we will overcome this via tools including the Burkholder-Davis-Gundy (BDG) inequality in Lemma \ref{lmm:zeta} below.
So far, to the best of our knowledge, there is scant literature addressing the ergodicity of discrete algorithms under such mild assumptions. These difficulties shall be addressed carefully using a series of conditional expectation estimates, detailed in Section \ref{section:omiited}.

The second difficulty arises from how to make use of the consistency of the random batch $\mathbb{E}_{\xi}[U^{\xi}(\cdot)]=U(\cdot)$ to prove the geometric ergodicity of SGLD. In our result, we only assumed the confining property for the expected potential. On each time subinterval $[T_k, T_{k+1})$ of the SGLD algorithm, one only sees the behavior of the process associated with $U^{\xi_k}$ rather than $U$. One has to consider the averaged dynamics so that our assumptions for the averaged potential $U$ can be used. So more technical details will be required to obtain the ergodicity, see Proposition \ref{pro:rbandconv}. Here we give a brief summary of Proposition \ref{pro:rbandconv} regarding the estimate for the random batch. Recall the paired dynamics $(\bar{X}_t, \bar{Y}_t)$ discussed above. After It\^o's calculation, one needs to estimate 
$$
\mathbb{E}\left[\phi(\bar{X}_t,\bar{Y}_t)(\nabla U^{\xi_k}(\bar{X}_{T_k}) - \nabla U^{\xi_k}(\bar{Y}_{T_k}))\right]
$$
for some function $\phi(\cdot,\cdot)$ and $t \in [T_k,T_{k+1})$. The key step is to use conduct the following splitting:
\begin{equation*}
\begin{aligned}
        &\quad\mathbb{E}\left[\phi(\bar{X}_t,\bar{Y}_t)(\nabla U^{\xi_k}(\bar{X}_{T_k}) - \nabla U^{\xi_k}(\bar{Y}_{T_k}))\right]\\
        &= \mathbb{E}\left[\phi(\bar{X}_{T_k},\bar{Y}_{T_k})(\nabla U^{\xi_k}(\bar{X}_{T_k}) - \nabla U^{\xi_k}(\bar{Y}_{T_k}))(\textbf{1}_A + \textbf{1}_B)\right] \\
         &\quad+ \mathbb{E}\left[(\phi(\bar{X}_{t},\bar{Y}_{t}) - \phi(\bar{X}_{T_k},\bar{Y}_{T_k})(\nabla U^{\xi_k}(\bar{X}_{T_k}) - \nabla U^{\xi_k}(\bar{Y}_{T_k}))(\textbf{1}_A + \textbf{1}_B)\right]
\end{aligned}
\end{equation*}
where $A := \{|\bar{X}_{T_k} - \bar{Y}_{T_k}| > R \}$ and $B:=A^c$ for some $R>0$. Since $\bar{X}_{T_k}$, $\bar{Y}_{T_k}$, $\textbf{1}_A$ are all independent of the random batch $\xi_k$, we are able to use the consistency of random batch $\mathbb{E}_{\xi}[U^{\xi}(\cdot)]=U(\cdot)$ and obtain the following for the first term:
\begin{multline*}
    \mathbb{E}\left[\phi(\bar{X}_{T_k},\bar{Y}_{T_k})(\nabla U^{\xi_k}(\bar{X}_{T_k}) - \nabla U^{\xi_k}(\bar{Y}_{T_k}))\right] = \mathbb{E}\left[\phi(\bar{X}_{T_k},\bar{Y}_{T_k})(\nabla U(\bar{X}_{T_k}) - \nabla U(\bar{Y}_{T_k}))\right] 
\end{multline*}
Actually, this equality above reveals the consistency between SGLD and the overdamped Langevin diffusion, since it remains true if we replace $\bar{Y}_t$ above with some solution to the overdamped Langevin diffusion.
Moreover, for the second term, under the event $A$, we use the uniform-in-batch Lipschitz condition in Assumption \ref{ass} to bound $\nabla U^{\xi_k}(\bar{X}_{T_k}) - \nabla U^{\xi_k}(\bar{Y}_{T_k})$, and the tail estimate obtained in Lemma \ref{lmm:zeta} below to estimate $(\phi(\bar{X}_{t},\bar{Y}_{t}) - \phi(\bar{X}_{T_k},\bar{Y}_{T_k})$. So eventually obtain an estimate for
$$
\mathbb{E}\left[\phi(\bar{X}_t,\bar{Y}_t)(\nabla U^{\xi_k}(\bar{X}_{T_k}) - \nabla U^{\xi_k}(\bar{Y}_{T_k}))\textbf{1}_A\right]
$$
under the event $A$ in Proposition \ref{pro:rbandconv}. For estimate under the event $B$, one cannot directly apply tail estimate in Lemma \ref{lmm:zeta} to estimate the remainder term in the splitting above. So the consistency of random batch is no longer used, but the uniform-in-batch Lipschitz condition in Assumption \ref{ass} is required in our derivation, scattered throughout the proof (see \eqref{eq:secondpart}, \eqref{eq:lmmabs} for instance). This uniform-in-batch Lipschitz condition is natural: Intuitively, the average of a family of non-smooth functions could be smooth, while in this case one cannot guarantee the convergence since the SGLD dynamics can evolve with non-smooth drift in all time.

The rest of the paper is organized as follows. In Section \ref{sec:assresults}, we list our main assumptions and main results. 
The detailed proof will be given in Section \ref{section:omiited}, where a series of key estimates for the conditional expectations will be given. Section \ref{sec:general} is for the generalization to drifts that are not necessarily gradients. In Appendix \ref{sec:lmmmoment}, some missing proofs will be given.

\section{Assumptions and main result}\label{sec:assresults}

\subsection{Local nonconvexity assumption}\label{sec:assumption}

We will use the reflection coupling to show the ergodicity under the following locally nonconvex setting, which is common in many practical tasks.

\begin{assumption}\label{ass}
    \begin{itemize}
        \item[(a)](locally nonconvex) The Hessian matrix of $U$ is uniformly positive definite outside $B(0,R_0)$, namely, there exist $R_0 >0$, $\kappa_0 > 0$ such that
        \begin{equation}
            \nabla^2 U(x) \succeq \kappa_0 I_d, \quad \forall x \in \mathbb{R}^d \setminus B(0,R_0);
        \end{equation}
        \item[(b)](global uniform-in-batch Lipshitz) There exists $K > 0$ such that $\forall x,y \in \mathbb{R}^d$, $\forall \xi \in \mathcal{S}$,
        \begin{equation}
            \left|\nabla U^{\xi}(x) - \nabla U^{\xi}(y) \right| \leq K |x-y|.
        \end{equation}
    Moreover, $\sup_{\xi}|\nabla U^{\xi}(0)|<\infty$.
    \end{itemize}
\end{assumption}

\begin{remark}
In many applications in data science, people study the empirical risk with a penalty \cite{mou2018generalization,farghly2021time}. Particularly, one may consider
\begin{equation*}
    \tilde{U} = \frac{1}{N} \sum_{i=0}^N \ell_i(x) + \frac{\lambda}{2} |x|^2.
\end{equation*}
If $\ell_i$'s have certain decay property as $|x|\to\infty$, the function $\tilde{U}$ satisfies Assumption \ref{ass}. For instance, one may consider $\tilde{U}$ being the cross-entropy loss with some additional $l^2$-penalty as in some machine learning tasks; one may also compare this with some analogous examples in the interaction particle systems, where one has suitable bounded interactions and some external force $U_0$ \cite{jin2020random,jin2022ergodicity}.
\end{remark}

From the locally nonconvex setting in Assumption \ref{ass}, it is not hard to derive the following strong monotonicity property for the pair $(x,y)$, which is useful in our analysis.
\begin{lemma}\label{lmm:ass}
Suppose Assumption \ref{ass} holds, then there exists $R\ge 2$, $\kappa > 0$ such that 
\begin{equation}\label{eq:lmmass}
    (x - y) \cdot (\nabla U(x) - \nabla U(y)) \geq \kappa |x - y|^2, \quad \forall x,y \in \mathbb{R}^d,\,|x-y|>R.
\end{equation}
\end{lemma}
The proof is deferred to Section \ref{sec:lmmmoment}.
Another useful observation from Assumption \ref{ass} is that, we are able to control $p$-th moment ($p>1$) of the SGLD iteration $X_t$ defined in \eqref{eq:sgld}. See the detailed proof in Section \ref{sec:lmmmoment}.

\begin{lemma}\label{lmm:sgldmoment}[Moment control for SGLD] Consider the SGLD iteration \eqref{eq:sgld}.
Suppose Assumption \ref{ass} holds.
\begin{enumerate}[(a)]
\item For any $p \ge 1$, any $T>0$ and any step size $\eta_k>0$,
        \begin{equation}\label{eq:lmm23}
            \sup_{0\leq t\leq T} \mathbb{E}\left[\sup_{0\leq s\leq t}\left|\bar{X}_s\right|^p\right] < +\infty.
        \end{equation}
        The upper bound may depend on $p$, $T$, $\beta$ and the dimension $d$.
\item Let $p\ge 2$. If $\exists \delta>0$ such that $\eta_k \le \kappa/(2(p-1)K^2)-\delta$ for all $k$, then
 \begin{equation}\label{eq:conclusion_a}
 \sup_{t\geq 0} \mathbb{E}\left|\bar{X}_t\right|^p < +\infty.
\end{equation}
        The upper bound may depend on $p$, $\beta$ and the dimension $d$.
\end{enumerate}
\end{lemma}

\subsection{Geometric ergodicity of SGLD}\label{sec:2-2}
By considering the behavior of $Z_t$, we aim to obtain a contraction result in terms of the Kantorovich-Rubinstein distance defined by:
\begin{equation}\label{eq:Wf}
    W_{f}(\mu, \nu):=\inf _{\gamma \in \Pi(\mu, \nu)} \int_{\mathbb{R}^{d} \times \mathbb{R}^{d}}f(|x-y|) d \gamma.
\end{equation}
One aims to find some suitable increasing, concave function $f$ such that $f(|\cdot|)$ is equivalent to $|\cdot|$ and hence one is able to control Wasserstein-1 distance using $W_f$. The specific function we consider in this work is given by
\begin{equation}\label{eq:f_sgld}
    f(r) := \int_0^r e^{-c_f (s \wedge R_1)} ds, \quad r \geq 0.
\end{equation}
Here, $R_1>3R/2$ and $c_f>0$ are constants to be determined. Clearly, $f$ is concave and increasing. Moreover, for $r \geq 0$,
\begin{equation}
    e^{-c_f R_1} r \leq f(r) \leq r.
\end{equation}
We will discuss more on the motivation of the construction for such paired dynamics $(\bar{X}_t, \bar{Y}_t)$ and the Lyapunov function in Section \ref{section:omiited}.

Next, we will take $c_f>0$ such that 
\begin{gather}\label{eq:condcf}
\frac{1}{2}\sqrt{2\beta^{-1}} c_f R^{-1} - K \ge 0
\end{gather}
and fix $R_1 := 2R$. Moreover, we consider small steps and require the upper bound for the time step $h := \sup_k \eta_k$ to satisfy 
\begin{gather}
\begin{split}
& h^{\frac{1}{2}} |\log h|^{\frac{1}{2}} \leq \min \left(\frac{1}{6 c'} e^{-2c_f R} \kappa, \sqrt{\bar{c}} \right),\\
& h \leq \bar{c}^{-\frac{1}{2}}h^{\frac{1}{2}} |\log h|^{\frac{1}{2}}(KR)^{-1},\quad h \le \min(1/(2K), R^2/9, 1)
\end{split}
\end{gather}
and
\begin{multline}\label{eq:condh3}
h \leq \min\Big( \frac{\bar{c}\beta R^2}{128 \log 5}, 
\frac{\bar{c}\beta R^2}{128\left(2c_f R + \log (18 / \kappa) \right)},
 \frac{\bar{c}\beta R^2 / 32}{\log \left( 45 \left(1 + \sqrt{2\beta}c_f^{-1}e^{2c_f R} K R \right)/2\right)}\Big),
\end{multline}
where $\bar{c}$, $c'$  are two positive constants independent of $k$ and the dimension $d$ coming from Lemma \ref{lmm:zeta} and Proposition \ref{pro:rbandconv} respectively.

We will establish in this work the following Wasserstein contraction results of SGLD. We leave the detailed proof to Section \ref{section:omiited}.

\begin{theorem}\label{thm:contraction}[Wasserstein contraction for SGLD]
Suppose Assumption \ref{ass} holds. For any two initial distributions $\mu_0$ and $\nu_0$, denote $\mu_t$ and $\nu_t$ to be the corresponding time marginal distributions for the time continuous interpolation of SGLD algorithm \eqref{eq:sgld}. Denote $h := \sup_k \eta_k$. Let $f$ be the Lyapunov function defined in \eqref{eq:f_sgld}.  Assume that $h$ and the parameters satisfy conditions in \eqref{eq:condcf}-\eqref{eq:condh3}, then the following Wasserstein contraction result holds:
\begin{equation}
    W_f(\mu_{T_k}, \nu_{T_k}) \leq e^{-cT_k}W_f(\mu_0,\nu_0), \quad k \in \mathbb{N},
\end{equation}
where
\begin{equation*}
c = \frac{1}{3}e^{-2c_f R}\min\left(\sqrt{2\beta^{-1}} c_f R/2, \kappa \right).
\end{equation*}
Consequently,
\begin{equation}
    W_1(\mu_{T_k}, \nu_{T_k}) \leq c_0 e^{-cT_k}W_1(\mu_0,\nu_0), \quad k \in \mathbb{N},\quad c_0 := e^{2c_f R}.
\end{equation}
\end{theorem}

The contraction rate is not necessarily the optimal one, which we believe is dimension-free (see the discussion in Remark \ref{rmk:d}). We have listed many restrictions on the step size. For the second restriction, the most essential one is that we need $h<1/K$ for the contraction to hold. Here, we required $h\le 1/(2K)$ instead for the formulas of contraction rate to be of reasonable order. Other restrictions on the step size can be relaxed somehow (for example $3R/2$ can be replaced by a number close to $R$ and the numerates are loose). They are chosen just to make the formula of contraction rate appear clean. However, the dependence of $\beta$ in the upper bound of the third restriction is essential. Besides, only  Lemma \ref{lmm:sgldmoment} (a) is needed for the proof, so the restriction in Lemma \ref{lmm:sgldmoment} (b) is not included.

Moreover, if the step size (or learning rate) is constant $\eta_k\equiv \eta$ such that the discrete chain is time-homogeneous, then the SGLD as a discrete time Markov chain has an invariant measure $\tilde{\pi}$ by the Banach contraction mapping theorem  \cite{lax2002functional}. In particular, we have the following corollary:
\begin{corollary}\label{coro:ergodicity}[Wasserstein ergodicity of SGLD]
Consider the SGLD with constant step size $\eta_k\equiv \eta$. Assume that the step size $\eta$ satisfies the restrictions in Theorem \ref{thm:contraction}, for any initial distribution $\rho_0 \in W_1$, the SGLD iteration has a unique invariant measure $\tilde{\pi}$, and the time marginal distribution $\bar{\rho}_t$ of \eqref{eq:sgld} satisfies for the constants $c_0, c$ in Theorem \ref{thm:contraction} that
\begin{equation}\label{W1ergodicity}
    W_1(\bar{\rho}_{n\eta}, \tilde{\pi}) \leq c_0 e^{-cn\eta}W_1(\rho_0, \tilde{\pi}).
\end{equation}
\end{corollary}

\begin{proof}
By Theorem \ref{thm:contraction}, there exists $k_0 \in \mathbb{N}_{+}$  such that 
\begin{equation}
    W_1(\mu_{T_{k_0}}, \nu_{T_{k_0}}) \leq \frac{1}{2} W_1(\mu_0,\nu_0).
\end{equation}
Denote the corresponding transition kernel for $n$th iteration by $P_{n}$. Then, $\mu \mapsto \mu P_{k_0}$ is contractive. By contraction mapping theorem, there exists a fixed point $\pi_*$ satisfying
\begin{equation}
    \pi_* = \pi_* P_{k_0}.
\end{equation}
Then, by Markov property, $\tilde{\pi} := \frac{1}{k_0}\sum_{n=0}^{k_0-1} \pi_* P_n$ is the invariant measure of the SGLD iteration. 
Moreover, $\tilde{\pi} = \tilde{\pi} P_{k_0}$ for any invariant measure so that the invariant measure is unique by the contraction property of $P_{k_0}$. Besides, $\tilde{\pi}=\pi_*$. 

Letting $\nu_{n\eta} \equiv \tilde{\pi}$ in Theorem \ref{thm:contraction}, \eqref{W1ergodicity} then follows.
\end{proof}

Under Assumption \ref{ass}, $\pi \propto e^{-U}$ satisfies the log-Sobolev inequality, and one can get a uniform-in-time error estimate using KL divergence in \cite{li2022sharp}. We are then able to estimate the $W_1$ distance between the target distribution $\pi$ and the invariant measure $\tilde{\pi}$ of the SGLD algorithm. In fact, for constant step size $\eta$, by \cite[Theorem 3.2]{li2022sharp}, the discretization error in terms of relative entropy (or KL-divergence) is given by 
\begin{equation}\label{eq:DKL}
    H(\bar{\rho}_{n\eta} || \rho_{n\eta}) \leq A_0 \eta^2, \quad \forall n \in  \mathbb{N},
\end{equation}
where $\bar{\rho}_{n\eta}$, $\rho_{n\eta}$ correspond to the SGLD iteration and the overdamped Langevin diffusions, respectively. As a remark, the constant $A_0$ scales almost linearly with the dimension $d$ under certain assumptions. The reason for the improved error bound in \eqref{eq:DKL} (from $O(\sqrt{\eta})$ to $O(\eta)$ in terms of Wasserstein or total vatiantion distance, in comparison with existing results like \cite{farghly2021time,mou2018generalization,zou2021faster,chau2021stochastic,zhang2023nonasymptotic}) is that, starting from the Fokker-Planck equation for the discrete algorithm, the authors directly considered the distance between distribution instead of other trajectory methods; also, techniques like Girsanov's transform were applied to handle additional difficulties brought by the random batch. As a consequence of \eqref{eq:DKL}, since $\pi$ satisfies the log-Sobolev inequality, the Wasserstein-1 distance can be controlled by the square root of the KL-divergence by some classical transportation inequalities \cite{otto2000generalization,talagrand1991new}, enabling one to derive an improved sampling error bound $W_1\left(\bar{\rho}_{n\eta}, \pi \right)$ for SGLD (\cite{li2022sharp}, Corollary 5.1). We conclude the result in the following corollary.
\begin{corollary}
Consider the SGLD with constant step size $\eta$ and denote its density at time $n\eta$ by $\bar{\rho}_{n\eta}$. Under Assumption \ref{ass}, for the step size $\eta$ small enough (with the restrictions in Theorem \ref{thm:contraction} and in Theorem 3.2 of \cite{li2022sharp}), for some positive $A$, $C_1$, $C_2$ independent of $\rho_0$, $\eta$ we have
\begin{equation}
    W_1\left(\bar{\rho}_{n\eta}, \pi \right) \leq C_0\eta + C_1 e^{-C_2 n\eta}.
\end{equation}
Moreover, the SGLD iteration has a unique invariant measure $\tilde{\pi}$ satisfying 
\begin{equation}\label{eq:pipi}
    W_1(\tilde{\pi}, \pi) \leq A\eta,
\end{equation}
where $\pi \propto e^{-\beta U}$ is the target distribution.
\end{corollary}


\section{Proof of the Theorem \ref{thm:contraction}}\label{section:omiited}

In this section, we prove Theorem \ref{thm:contraction} - the contraction property under the $W_f$ distance.
In the following, we will apply the technique of reflection coupling discussed in the introduction to analyze SGLD. See Appendix \ref{app:detail} for more details on the construction of the reflection coupling and the Lyapunov function.

We summarize here several challenges we would overcome in the analysis. The first difficulty arises from how to make use of the consistency of the random batch $\mathbb{E}_{\xi}[U^{\xi}(\cdot)]=U(\cdot)$ to prove the geometric ergodicity of SGLD. In our result, we only assumed the confining property for the expected potential. On each time subinterval $[T_k, T_{k+1})$ of the SGLD algorithm, one only sees the behavior of the process associated with $U^{\xi_k}$ rather than $U$. Therefore, one has to consider the averaged dynamics so that our assumptions for the averaged potential $U$ can be used. So more technical details will be required to obtain the ergodicity, see Proposition \ref{pro:rbandconv} below. Secondly, we look into the issues that come with numerical discretization - given the discrete nature of the scheme, the drift term for SGLD is evaluated at $X_{T_k}$ instead of $X_t$, introducing additional challenging elements into our analysis. Furthermore, when dealing with this difficulty coming from time discretization, one needs to carefully estimate the tail behavior of the multiplicative noise $\zeta_t$ in \eqref{eq:zetaj_def}. In fact, although the diffusion part $\hat{W}_t$ in \eqref{eq:reflectioncoupling} or Lemma \ref{lmm:verifyBM} is a Brownian motion, it correlates with the original Brownian motion $W_t$. Therefore, in \eqref{eq:Ztevolve}, $\zeta_t = \int_{T_k}^t d\hat{W}_s - dW_s$ is a multiplicative noise. Estimate for this multiplicative noise is not trivial and we will overcome this via tools including the Burkholder-Davis-Gundy (BDG) inequality in Lemma \ref{lmm:zeta} below.


\subsection{Reflection coupling for SGLD}\label{sec:reflection}

For any two initial distributions $\mu_0$, $\nu_0$ in the statement of Theorem \ref{thm:contraction}, we construct the following reflection coupling:
\begin{equation}\label{eq:reflectioncoupling}
    \begin{aligned}
        & d\bar{X}_t = -\nabla U^{\xi_k} (\bar{X}_{T_k}) dt + \sqrt{2\beta^{-1}} dW, \quad t\in [T_k,T_{k+1}), \quad t < \tau;\\
        & d\bar{Y}_t = -\nabla U^{\xi_k} (\bar{Y}_{T_k})dt + \sqrt{2\beta^{-1}} \left(I_d - 2 e_t \otimes e_t \right) \cdot dW, \quad t \in [T_k,T_{k+1}), \quad t < \tau;\\
        & \bar{X}_t = \bar{Y}_t,\quad t \geq \tau,
    \end{aligned}
\end{equation}
where 
\begin{equation}
    e_t := \frac{\bar{X}_t - \bar{Y}_t}{|\bar{X}_t - \bar{Y}_t|},
\end{equation}
and the stopping time $\tau$ is defined by 
\begin{equation}\label{eq:deftau}
    \tau:=\inf \{t\geq 0: \bar{X}_t = \bar{Y}_t \}.
\end{equation}
Moreover, the initials $\bar{X}_0$, $\bar{Y}_0$ of \eqref{eq:reflectioncoupling} should be chosen such that 
\begin{equation}\label{eq:optimalinitialcouple}
    \mathbb{E}f(|\bar{X}_0 - \bar{Y}_0|) = W_f(\mu_0,\nu_0).
\end{equation}
Recall the definition of $W_f$ in \eqref{eq:Wf}. For any two $\mu_0$, $\nu_0$ in Theorem \ref{thm:contraction}, \eqref{eq:optimalinitialcouple} can actually be achieved since one can always choose an optimal coupling $\gamma \in \pi(\mu_0,\mu_0)$ such that $\int_{\mathbb{R}^d \times \mathbb{R}^d} f(|x-y|) d\gamma =  W_f(\mu_0,\nu_0)$ \cite{villani2009optimal}, and in this case, $\bar{X}_0 \sim \mu_0$ and $\bar{Y}_0 \sim \nu_0$.

Note that $\int_0^{t}(I_d - 2\textbf{1}_{\{s<\tau \}} e_s \otimes e_s)\cdot dW_s$ is also a Brownian motion. Then, $\bar{Y}_t$ is thus also a copy of the time continuous interpolation of SGLD. Therefore, \eqref{eq:reflectioncoupling} is a well-defined coupling for the SGLD iteration.
Similar arguments also appeared in related literature like \cite{eberle2011reflection, eberle2016reflection}. We summarize this in the following Lemma:
\begin{lemma}\label{lmm:verifyBM}
Under the settings of \eqref{eq:reflectioncoupling} and \eqref{eq:deftau}, the process
\begin{equation*}
    \hat{W}_t := \int_0^t (I_d - 2 \textbf{1}_{\{s<\tau \}}e_s  e_s^T)  dW_s
\end{equation*}
is a Brownian motion in $\mathbb{R}^d$ with respect to the natural filtration. Consequently, $\bar{Y}_t$ is also a realization of SGLD \eqref{eq:sgld}.
\end{lemma}
\begin{proof}
Clearly, $\hat{W}_0 = 0$ and $\hat{W}_t$ is a martingale with respect to $\mathcal{F}_t := \sigma (W_s:s\leq t)$. Then by Levy's characterization of Brownian motion, one only needs to verify that for any $t' > t > 0$, $\mathbb{E}[\hat{W}_{t'} \otimes \hat{W}_t] = tI_d$. Indeed, by independent increment of the Brownian motion $W_t$, one has
\begin{equation*}
\begin{aligned}
    \mathbb{E}\left[\hat{W}_{t'} \otimes \hat{W}_t\right] &= \mathbb{E}\left[\left(\int_0^t (I_d - 2\textbf{1}_{\{s<\tau \}} e_s  e_s^T)  dW_s\right)\left(\int_0^t (I_d - 2 \textbf{1}_{\{s<\tau \}}e_s  e_s^T)  dW_s\right)^T\right]\\
    &=\int_0^t \mathbb{E}\left[(I_d - 2 \textbf{1}_{\{s<\tau \}}e_s  e_s^T)(I_d - 2\textbf{1}_{\{s<\tau \}} e_s  e_s^T)^T\right] ds = tI_d,
\end{aligned}
\end{equation*}
where the last inequality is due to the fact that 
\begin{equation*}
    e_s^T e_s = \frac{\left(\bar{X}_s - \bar{Y}_s\right)^T}{|\bar{X}_s - \bar{Y}_s|} \frac{\left(\bar{X}_s - \bar{Y}_s\right)}{|\bar{X}_s - \bar{Y}_s|} = 1, \quad \forall s \geq 0.
\end{equation*}
Therefore, the process $\hat{W}_t$ is a Brownian motion in $\mathbb{R}^d$. Consequently, $\bar{Y}_t$ is also a solution of SGLD \eqref{eq:sgld}.
\end{proof}

Denote $Z_t := \bar{X}_t - \bar{Y}_t$. Then for $t\in[T_k,T_{k+1})$ and $t<\tau$,  the process $Z$ satisfies
\begin{equation}\label{eq:1_5}
    dZ_t = -\left(\nabla U^{\xi_k} (\bar{X}_{T_k}) - \nabla U^{\xi_k} (\bar{Y}_{T_k}) \right) dt + 2\sqrt{2\beta^{-1}} \frac{Z_t^{\otimes 2}}{|Z_t|^2} \cdot dW,
\end{equation}
and $Z_t = 0$ for all $t\geq \tau$.

Clearly, the process $Z_t$ defined in \eqref{eq:1_5} satisfies for $t \in [T_k,T_{k+1})$, 
\begin{equation}\label{eq:Ztevolve}
    Z_t = Z_{T_k} - (t\wedge \tau - T_k\wedge \tau)\, A_{T_k} + 2\sqrt{2\beta^{-1}}\zeta_t,
\end{equation}
where the process $\zeta_t$ is defined by 
\begin{equation}\label{eq:zeta_def}
    \zeta_t := \int_{{T_k} \wedge \tau}^{t \wedge \tau} \frac{Z_s^{\otimes 2}}{|Z_s|^2} \cdot dW_s ,
\end{equation}
and 
\begin{equation}\label{eq:ATKdef}
    A_{T_k} := \nabla U^{\xi_k}(\bar{X}_{T_k}) - \nabla U^{\xi_k}(\bar{Y}_{T_k}).
\end{equation}
Clearly, by optional stopping theorem \cite{durrett2018stochastic}, $\zeta_t$ is a martingale. Later in Lemma \ref{lmm:zeta} and Corollary \ref{lmm:abs_zeta}, we will prove some sub-Guassian properties of such martingale. These estimates are very helpful to overcome the challenge brought by numerical discretization. We remark that  \eqref{eq:Ztevolve} and \eqref{eq:zeta_def} also guarantee that $Z_t \equiv Z_{T_k} = 0$ for $t\ge T_k \geq \tau$, which is consistent with the definition of the coupling.

\subsection{Geometric ergodicity and uniform estimate for  SGLD}\label{sec:sec32}
Recall that the increasing, concave function $f$ is of the form
\begin{equation}
    f(r) = \int_0^r e^{-c_f (s \wedge R_1)} ds, \quad r \geq 0.
\end{equation}
With the construction in \eqref{eq:reflectioncoupling}, we are then able to prove the geometric ergodicity of SGLD. In fact, due to the argument at the beginning of Section \ref{section:omiited}, we aim to show that
\begin{equation*}
    \mathbb{E}f(|Z_t|) \leq e^{-ct}\mathbb{E}f(|Z_0|),
\end{equation*}
which is clearly equivalent to
\begin{equation}\label{eq:original_goal}
    \mathbb{E}f(|Z_{t \wedge \tau}|) \leq e^{-ct}\mathbb{E}f(|Z_0|),
\end{equation}
Introduce the regularization stopping time sequence
\begin{gather}\label{eq:tauj}
\tau_j := \inf\{t\geq 0: |Z_t| \notin (j^{-1},j) \},
\quad j \in \mathbb{N}_{+},
\end{gather}
which is increasing and can be proved to converge to $\tau$ as $j \rightarrow \infty$ later in Lemma \ref{lmm:tauj}. Hence to obtain \eqref{eq:original_goal}, by Fatou's Lemma, one needs to show that
\begin{equation*}
        \mathbb{E}f(|Z_{t \wedge \tau_j}|) \leq e^{-ct}\mathbb{E}f(|Z_0|).
\end{equation*}
Therefore, the main goal in the proof of Theorem \ref{thm:contraction} is to give the following uniform estimate:
\begin{equation}
    \frac{d}{dt} \mathbb{E} f(|Z_{t \wedge \tau_j}|) \leq -c \mathbb{E} f(|Z_{t \wedge \tau_j}|),
\end{equation}
where $c$ is independent of $j$, $\eta_k$ and $\xi_k$.

In the following, we give the proof of our main result, Theorem \ref{thm:contraction}. Some auxiliary lemmas and their proofs will be given in Section \ref{sec:proplmm}.

\begin{proof}[Proof of Theorem \ref{thm:contraction}]
Recall the definition of $\tau_j$ in \eqref{eq:tauj}.
We first fix $T>0$ and consider those $k$ values such that $T_{k+1}\le T$.
Consider the process $Z_t^{\tau_j}:= Z_{t\wedge \tau_j}$.  Clearly, for $t \in [T_k,T_{k+1})$,
\begin{equation}\label{eq:Zttauj}
    Z_t^{\tau_j} = Z_{T_k}^{\tau_j} - (t\wedge \tau_j - T_k\wedge \tau_j)\, A_{T_k} + 2\sqrt{2\beta^{-1}}\zeta_t^{\tau_j},
\end{equation}
with
\begin{equation}\label{eq:zetaj_def}
    \zeta_t^{\tau_j} := \int_{{T_k} \wedge \tau_j}^{t \wedge \tau_j} \frac{Z_s^{\otimes 2}}{|Z_s|^2} \cdot dW_s.
\end{equation}
In fact, $\tau_j\le \tau$. If $\tau_j\le T_k$, $Z_t^{\tau_j}=Z_{T_k}^{\tau_j}$ and one can focus on the previous subinterval. If $\tau_j\in [T_k, T_{k+1})$, one may verify that this holds.

Corresponding to \eqref{eq:Zttauj}, the process $Z_t$ satisfies for $t \in (T_k\wedge \tau_j,T_{k+1}\wedge \tau_j)$,
\begin{equation*}
    dZ_t = -\left(\nabla U^{\xi_k} (\bar{X}_{T_k}) - \nabla U^{\xi_k} (\bar{Y}_{T_k}) \right) dt + 2\sqrt{2\beta^{-1}} \frac{Z_t^{\otimes 2}}{|Z_t|^2} \cdot dW.
\end{equation*}
Since
\begin{equation*}
    \nabla^2 f(|x|) = f''(|x|)\frac{x\otimes x}{|x|^2} + \frac{f'(|x|)}{|x|}\left(I_d - \frac{x \otimes x}{|x|^2} \right),
\end{equation*}
by Dykin's formula and the strong Markov property \cite{durrett2018stochastic},  one has then for $t \in [T_k,T_{k+1})$,
\begin{equation}\label{eq:dEdt}
    \frac{d}{dt} \mathbb{E} \left[ f(|Z_t^{\tau_j}|) \right] = \mathbb{E}\left[\left(\sqrt{2\beta^{-1}}f''(|Z_t^{\tau_j}|) - f'(|Z_t^{\tau_j}|) \frac{Z_t^{\tau_j}}{|Z_t^{\tau_j}|} \cdot A_{T_k}\right) \textbf{1}_{\{t < \tau_j \}}\right].
\end{equation}

Our goal is to obtain an upper bound of the form $-c\mathbb{E}\left[f(|Z_{t}^{\tau_j}|)\right]$ of the right hand side of \eqref{eq:dEdt}, where $A_{T_k}$ is computed using $U^{\xi}(X_{T_k})$ and $U^{\xi}(Y_{T_k})$. Note that we only assume convexity property for $U$ outside $B(0, R)$ as stated in Lemma \ref{lmm:ass}, so we firstly split the expectation in \eqref{eq:dEdt} into the following three parts
\begin{equation}\label{eq:aftersplit}
\begin{aligned}
    &\frac{d}{dt} \mathbb{E} \left[ f(|Z_t^{\tau_j}|) \right] = \mathbb{E}\left[-f'(|Z_t^{\tau_j}|) \frac{Z_t^{\tau_j}}{|Z_t^{\tau_j}|} \cdot A_{T_k} \textbf{1}_{\{|Z_{T_k}^{\tau_j}| > R \}}\textbf{1}_{\{t < \tau_j \}}\right] \\
    &+ \left(\mathbb{E}\left[ \sqrt{2\beta^{-1}}f''(|Z_t^{\tau_j}|)\textbf{1}_{\{t < \tau_j \}}\right] 
    + \mathbb{E}\left[-f'(|Z_t^{\tau_j}|) \frac{Z_t^{\tau_j}}{|Z_t^{\tau_j}|} \cdot A_{T_k} \textbf{1}_{\{|Z_{T_k}^{\tau_j}| \leq R \}}\textbf{1}_{\{t < \tau_j \}}\right]\right)\\
    & =: I_1(t) + I_2(t).
\end{aligned}
\end{equation}

Moreover, when estimating the right hand side of \eqref{eq:aftersplit}, we need to further split them like in Taylor's expansion. The essence of this splitting lies in two main actions: (1) conduct evaluations at $t = T_k$, which enables the utilization of the tower property by taking expectation with respect to the random batch $\xi_k$ first; (2) estimate the residual terms carefully. The main reason for such splitting is that, for $t \in [T_k,T_{k+1})$, $Z_t$ depends on the random batch $\xi_k$; for instance, $\mathbb{E}_{\xi_k}\left[Z_t \cdot \left(\nabla U^{\xi_k} (\bar{X}_{T_k}) - \nabla U^{\xi_k} (\bar{Y}_{T_k}) \right)\right] \neq \mathbb{E}_{\xi_k}\left[Z_t \cdot \left(\nabla U (\bar{X}_{T_k}) - \nabla  (\bar{Y}_{T_k}) \right) \right]$. With this main idea of splitting, in the following we will separately estimate each term in \eqref{eq:aftersplit}, and we will take $j$ sufficiently large such that  Lemma \ref{lmm:smallprobcontrol} (serving Proposition \ref{pro:rbandconv} below), Lemma \ref{lmm:lmm2} (serving the term $I_2(t)$) and Proposition \ref{pro:rbandconv} (serving the term $I_1(t)$) in Section \ref{sec:proplmm} below hold.


For the term $I_1(t)$, we will make use of the convexity condition along with the tail estimate in Lemma \ref{lmm:zeta}. By Proposition \ref{pro:rbandconv} (which is based on Lemma \ref{lmm:zeta}), for the step size $\eta_k $ in the range considered, then for $t \in [T_k,T_{k+1})$
\begin{equation}\label{eq:I3estimate}
    I_1(t) \leq - \left(e^{-c_f R_1}\kappa - c' \eta_k^{\frac{1}{2}}|\log \eta_k|^{\frac{1}{2}}-3e^{-\bar{c}\beta R^2 \eta_k^{-1}/128}  \right) \mathbb{E}\left[|Z_{T_k}^{\tau_j}| \textbf{1}_{\{|Z_{T_k}^{\tau_j}| > R \}} \right].
\end{equation}
Since $|Z_t^{\tau_j}| \leq |Z_{T_k}^{\tau_j}| + \eta_k K |Z_{T_k}^{\tau_j}| + |\zeta_t^{\tau_j}|$ by Assumption \ref{ass}, we have
\begin{gather*}
    -\mathbb{E}\left[|Z_{T_k}^{\tau_j}| \textbf{1}_{\{|Z_{T_k}^{\tau_j}| > R \}} \right]
    \leq -\frac{1}{1 + \eta_k K}\left(\mathbb{E}\left[|Z_{t}^{\tau_j}| \textbf{1}_{\{|Z_{T_k}^{\tau_j}| > R \}} \right] - \mathbb{E}\left[|\zeta_t^{\tau_j}| \textbf{1}_{\{|Z_{T_k}^{\tau_j}| > R \}} \right] \right).
\end{gather*}
Clearly, 
\begin{equation}
     \mathbb{E}\left[|\zeta_t^{\tau_j}| \textbf{1}_{\{|Z_{T_k}^{\tau_j}| > R \}} \right] \leq \sqrt{\eta_k} \mathbb{P}\left(|Z_{T_k}^{\tau_j}|>R\right)
     \le \sqrt{\eta_k}R^{-1}\mathbb{E}\left[|Z_{T_k}^{\tau_j}| \textbf{1}_{\{|Z_{T_k}^{\tau_j}| > R \}} \right].
\end{equation}
Therefore, for the event $\{|Z_{T_k}^{\tau_j}|>R \}$, one has
\begin{equation*}
\quad-\mathbb{E}\left[|Z_{T_k}^{\tau_j}| \textbf{1}_{\{|Z_{T_k}^{\tau_j}| > R \}} \right]
    \leq -\frac{1}{(1+\sqrt{\eta_k}R^{-1})(1+\eta_k K)}\mathbb{E}\left[|Z_{t}^{\tau_j}| \textbf{1}_{\{|Z_{T_k}^{\tau_j}| > R \}} \right].
\end{equation*}
This then implies for $\eta_k \le \min(1/(2K), R^2/9)$ that
\begin{equation}\label{eq:eqq332}
    -\mathbb{E}\left[|Z_{T_k}^{\tau_j}| \textbf{1}_{\{|Z_{T_k}^{\tau_j}| > R \}} \right] \leq
    -\frac{1}{2}\mathbb{E}\left[|Z_{t}^{\tau_j}| \textbf{1}_{\{|Z_{T_k}^{\tau_j}| > R \}} \right]
    \le -\frac{1}{2}\mathbb{E}\left[|Z_{t}^{\tau_j}| \textbf{1}_{\{|Z_{T_k}^{\tau_j}| > R \}}\textbf{1}_{\{t < \tau_j \}} \right],
\end{equation}
for large $j$.

Now combining \eqref{eq:I3estimate} and \eqref{eq:eqq332}, we have the following estimate for the term $I_1(t)$:
\begin{equation}\label{eq:I3final}
    I_1(t) \leq  - \frac{1}{2}\left(e^{-c_f R_1}\kappa - c' \eta_k^{\frac{1}{2}}|\log \eta_k|^{\frac{1}{2}}-3e^{-\frac{\bar{c}\beta R^2}{128\eta_k}}\right) \mathbb{E}\left[|Z_{t}^{\tau_j}| \textbf{1}_{\{|Z_{T_k}^{\tau_j}| > R \}}\textbf{1}_{\{t < \tau_j \}} \right].
\end{equation}

For $I_2(t)$, the strategy is to make use of $f''$ that provides a negative part when $|Z_t|$ is small. This can also be understood as utilization of the convexity from the concave function $f$ when $U$ does not has convexity in $B(0,R)$.

The first part in $I_2(t)$ is given using the definition by
\begin{equation}\label{eq:firstpart}
\mathbb{E}\left[ \sqrt{2\beta^{-1}}f''(|Z_t^{\tau_j}|)\textbf{1}_{\{t < \tau_j \}}\right]=\mathbb{E}\left[-\sqrt{2\beta^{-1}} c_f e^{-c_f |Z_t^{\tau_j}|} \textbf{1}_{\{|Z_t^{\tau_j}| \leq R_1 \}} \textbf{1}_{\{t < \tau_j \}}\right].
\end{equation}
For the second part in $I_2(t)$, by Assumption \ref{ass}, Lemma \ref{lmm:ass} and definition of $f$, one has
\begin{equation}\label{eq:secondpart}
    \begin{aligned}
    \quad &\mathbb{E}\left[-f'(|Z_t^{\tau_j}|) \frac{Z_t^{\tau_j}}{|Z_t^{\tau_j}|} \cdot A_{T_k} \textbf{1}_{\{|Z_{T_k}^{\tau_j}| \leq R \}}\textbf{1}_{\{t < \tau_j \}}\right]
    \leq \mathbb{E}\left[Kf'(|Z_t^{\tau_j}|)|Z_{T_k}^{\tau_j}| \textbf{1}_{\{|Z_{T_k}^{\tau_j}| \leq R \}}\textbf{1}_{\{t < \tau_j \}}\right] \\
    &\le
    \mathbb{E}\left[Ke^{-c_f |Z_t^{\tau_j}|}|Z_{T_k}^{\tau_j}|\textbf{1}_{\{|Z_{T_k}^{\tau_j}| \leq R, \, |Z_t^{\tau_j}| \leq R_1 \}}\textbf{1}_{\{t < \tau_j \}} \right]
     + \mathbb{E}\left[Ke^{-c_f R_1} |Z_{T_k}^{\tau_j}| \textbf{1}_{\{|Z_{T_k}^{\tau_j}| \leq R, \, |Z_t^{\tau_j}| > R_1 \}} \textbf{1}_{\{t < \tau_j \}}\right]
    \end{aligned}
\end{equation}

In principle, the idea is to use \eqref{eq:firstpart} to control the terms arising from \eqref{eq:secondpart}. Hence, one may get
\begin{equation*}
    \begin{aligned}
       &I_2(t)
        \leq \mathbb{E}\left[\left(-\frac{1}{2} \sqrt{2\beta^{-1}} c_f e^{-c_f |Z_t^{\tau_j}|} + K e^{-c_f |Z_t^{\tau_j}|} |Z_{T_k}^{\tau_j}| \right) \textbf{1}_{\{|Z_{T_k}^{\tau_j}| \leq R, |Z_t^{\tau_j}| \leq R_1 \}}\textbf{1}_{\{t < \tau_j \}} \right]\\
        & + e^{-c_f R_1}\left( - \frac{1}{2} \sqrt{2\beta^{-1}} c_f \, \mathbb{P}\left(|Z_t^{\tau_j}| \leq R_1, |Z_{T_k}^{\tau_j}| \leq R ,t<\tau_j\right)
        +KR\, \mathbb{P}\left(|Z_{T_k}^{\tau_j}| \leq R, |Z_t^{\tau_j}| > R_1,t<\tau_j\right)\right)\\
        &=: J_1(t)+J_2(t).
    \end{aligned}
\end{equation*}

Direct estimate yields:
\begin{equation}\label{J_1estimate}
\begin{aligned}
    J_1(t) &\leq \mathbb{E}\left[-e^{-c_f |Z_t^{\tau_j}|} \left( \frac{1}{2}\sqrt{2\beta^{-1}}c_f R^{-1} - K\right) |Z_{T_k}^{\tau_j}| \textbf{1}_{\{|Z_{T_k}^{\tau_j}| \leq R, |Z_t^{\tau_j}| \leq R_1 \}}\textbf{1}_{\{t < \tau_j \}}\right]\le 0,
\end{aligned}
\end{equation}
if 
\[
\frac{1}{2}\sqrt{2\beta^{-1}} c_f R^{-1} - K \ge 0.
\]
Here we need to choose sufficiently large coefficient $c_f$ in the definition of $f$ such that $\frac{1}{2}\sqrt{2\beta^{-1}} c_f R^{-1} - K \ge 0$.

To handle the remaining term $J_2(t)$, we first observe that
\begin{equation}\label{eq:J2_1}
\begin{aligned}
    J_2(t) &\leq -\frac{1}{2}\sqrt{2\beta^{-1}}c_f e^{-c_f R_1}R_1^{-1}\mathbb{E}\left[|Z_t^{\tau_j}| \textbf{1}_{\{ |Z_t^{\tau_j}| \leq R_1, |Z_{T_k}^{\tau_j}| \leq R ,t<\tau_j\}}\right]\\
    &\quad + KRR_1^{-1}\mathbb{E}\left[|Z_{t}^{\tau_j}| \textbf{1}_{\{|Z_t^{\tau_j}| > R_1, |Z_{T_k}^{\tau_j}| \leq R, t<\tau_j\}}\right]\\
    &= -\frac{1}{2}\sqrt{2\beta^{-1}}c_f e^{-c_f R_1}R_1^{-1}\mathbb{E}\left[|Z_t^{\tau_j}| \textbf{1}_{\{  |Z_{T_k}^{\tau_j}| \leq R ,t<\tau_j\}}\right]\\
    &\quad + \left(\frac{1}{2}\sqrt{2\beta^{-1}}c_f e^{-c_f R_1}R_1^{-1} +  KRR_1^{-1} \right) \mathbb{E}\left[|Z_{t}^{\tau_j}| \textbf{1}_{\{|Z_t^{\tau_j}| > R_1, |Z_{T_k}^{\tau_j}| \leq R, t<\tau_j\}}\right]
\end{aligned}
\end{equation}
Based on tail estimate in Lemma \ref{lmm:zeta}, we prove in Lemma \ref{lmm:lmm2}, for small $\eta_k$ and large $j$,
\begin{equation}\label{eq:secondclaim}
    \mathbb{E}\left[|Z_{t}^{\tau_j}| \textbf{1}_{\{|Z_t^{\tau_j}| > R_1, |Z_{T_k}^{\tau_j}| \leq R, t<\tau_j\}}\right] \leq \varepsilon(\eta_k) \mathbb{E}\left[|Z_{t}^{\tau_j}| \textbf{1}_{\{ |Z_{T_k}^{\tau_j}| \leq R, t<\tau_j\}}\right],
\end{equation}
where
\begin{equation}
    \varepsilon(\eta) := \frac{15R_1 e^{-\bar{c}\beta(R_1 - 3R/2)^2 \eta^{-1}/ 8}}{R}.
\end{equation}
Therefore, for the conditions given, 
\begin{equation}\label{eq:J_2estimate}
    J_2(t) \leq -\left(\frac{1}{2}\sqrt{2\beta^{-1}}c_f e^{-c_f R_1}R_1^{-1} - \tilde{\varepsilon}(\eta_k) \right)\mathbb{E}\left[|Z_t^{\tau_j}| \textbf{1}_{\{|Z_{T_k}^{\tau_j}| \leq R \}} \textbf{1}_{\{t < \tau_j \}}\right].
\end{equation}
where
\begin{equation*}
    \tilde{\varepsilon}(\eta) := \left(\frac{1}{2}\sqrt{2\beta^{-1}}c_f e^{-c_f R_1}R_1^{-1} +  KRR_1^{-1} \right)\frac{15R_1 e^{-\bar{c}\beta(R_1 - 3R/2)^2 \eta^{-1}/ 8}}{R}.
\end{equation*}

Hence, for $t \in [T_k,T_{k+1})$, one is able to conclude from \eqref{eq:I3final},   \eqref{eq:J_2estimate} that
\begin{equation*}
\begin{aligned}
    \frac{d}{dt}\mathbb{E}\left[f(|Z_t^{\tau_j}|)\right]
    \leq -c(k)\mathbb{E}\left[|Z_{t}^{\tau_j}|\textbf{1}_{\{t < \tau_j \}}\right]\leq -c(k)\mathbb{E}\left[f(|Z_{t}^{\tau_j}|)\textbf{1}_{\{t < \tau_j \}}\right],
\end{aligned}
\end{equation*}
where 
\begin{equation*}
c(k) := \min\left(  \frac{1}{2}\sqrt{2\beta^{-1}}c_f e^{-c_f R_1}R_1^{-1} - \tilde{\varepsilon}(\eta_k) ,\frac{1}{2}\left(e^{-c_f R_1}\kappa - c' \eta_k^{\frac{1}{2}}|\log \eta_k|^{\frac{1}{2}}-3e^{-\bar{c}\beta R^2 \eta_k^{-1}/128} \right) \right).
\end{equation*}

Letting $j \rightarrow + \infty$, since $\tau_j \to \tau$  by Lemma \ref{lmm:tauj} in Section \ref{sec:proplmm} and the moment control in \eqref{eq:Zcontrol} (recall that $T_{k+1}\le T$ and $Z_t=Z_{t}^{\tau}$), one has by dominated convergence theorem that
\begin{equation*}
\mathbb{E}\left[f(|Z_t|)\right]
    \leq \mathbb{E}\left[f(|Z_{T_k}|)\right]-c(k)\int_{T_k}^t \mathbb{E}\left[f(|Z_{s}|)\textbf{1}_{\{s < \tau \}}\right]\,ds.
\end{equation*}
Since $Z_t \equiv 0$ for $t \geq \tau$, 
\[
\mathbb{E}\left[f(|Z_{t}|)\textbf{1}_{\{t < \tau \}}\right]=\mathbb{E}\left[f(|Z_{t}|)\right].
\]
By Gr\"onwall's inequality, one has
\begin{equation*}
    \mathbb{E}\left[f(|Z_{T_{k+1}}|)\right]  \leq e^{-c(k)\eta_k}\mathbb{E}\left[f(|Z_{T_{k}}|)\right].
\end{equation*}
Define $h := \sup_k \eta_k$.
Choosing small $h$ as stated in Theorem \ref{thm:contraction}, we are able to
conclude
\begin{equation}
 \mathbb{E}\left[f(|Z_{T_{k}}|)\right]
 \le e^{-c T_k} \mathbb{E}\left[f(|Z_0|)\right] =e^{-cT_k}W_f(\mu_0,\nu_0),
\end{equation}
where
\begin{equation}
    c = \frac{1}{3}e^{-c_f R_1}\min\left(\sqrt{2\beta^{-1}} c_f R_1^{-1}, \kappa \right).
\end{equation}
Note that this resulted inequality is independent of $T$. Since $T$ is arbitrary, this holds for all $k\ge 0$. So one eventually has
\begin{equation}\label{eq:contractionoT}
    W_f(\mu_{T_k},\nu_{T_k}) \leq e^{-cT_k}W_f(\mu_0,\nu_0), \quad \forall k \in \mathbb{N}.
\end{equation}
Above, $\mu_0$ and $\nu_0$ denote any two initial distributions, and $\mu_t$ and $\nu_t$ are the corresponding time marginal distributions for the time continuous interpolation of SGLD algorithm \eqref{eq:sgld}.   Moreover, since $e^{-c_f R_1} r \leq f(r) \leq r$ for any $r \geq 0$, we have
\begin{equation}
    W_1(\mu_{T_k}, \nu_{T_k}) \leq c_0 e^{-cT_k}W_1(\mu_0,\nu_0), \quad \forall k \in \mathbb{N}.
\end{equation}
with $c_0 := e^{c_f R_1}$. This then ends the proof by choosing $R_1 = 2R$.
\end{proof}

\subsection{Propositions and Lemmas used in the proof of Theorem \ref{thm:contraction}}\label{sec:proplmm}

In the following, we present crucial estimations used in the proof above: Lemma \ref{lmm:zeta}, Corollary \ref{lmm:abs_zeta}, Lemma \ref{lmm:smallprobcontrol}, Lemma \ref{lmm:lmm2}, Proposition \ref{pro:rbandconv} and Lemma \ref{lmm:tauj}. As can be observed in the proof of Theorem \ref{thm:contraction}, the main issue to solve due to the existence of numerical discretization is that, one needs to estimate how far $Z_t$ moves during the time interval $[T_k,t]$ for $t \in [T_k,T_{k+1})$. This then motivates us to estimate the diffusion part $\zeta_t^{\tau_j}$ defined in \eqref{eq:zetaj_def} first.

In the next Lemma, we estimate the martingale $\zeta_t^{\tau_j}$ defined in \eqref{eq:zetaj_def}:
\begin{equation*}
    \zeta_t^{\tau_j} := \int_{{T_k} \wedge \tau_j}^{t \wedge \tau_j} \frac{Z_s^{\otimes 2}}{|Z_s|^2} \cdot dW_s.
\end{equation*}
Note that although the diffusion part $\hat{W}_t$ in \eqref{eq:reflectioncoupling} or Lemma \ref{lmm:verifyBM} is a Brownian motion, it correlates with the original Brownian motion $W_t$. Therefore, in \eqref{eq:Ztevolve}, $\zeta_t = \int_{T_k}^t d\hat{W}_s - dW_s$ is a multiplicative noise.  However, since the diffusion coefficient $\frac{Z_t^{\otimes 2}}{|Z_t|^2}$ has unit norm, we are able to give the following tail estimate for $\zeta_t^{\tau_j}$ using the Burkholder-Davis-Gundy (BDG) inequality.

\begin{lemma}\label{lmm:zeta}
Recall the definition for $\zeta_s^{\tau_j}$ in \eqref{eq:zetaj_def}. For any $j \in \mathbb{N}_{+}$, for fixed $t \in [T_k, T_{k+1})$, the random variable $\sup_{T_k\le s\le t}|\zeta_s^{\tau_j}|$ is subgaussian in the sense that
\begin{equation}\label{eq:subguassianprop}
    \mathbb{P}\left(\sup_{T_k \leq s \leq t}|\zeta_s^{\tau_j}| > a\Big| \mathcal{F}_{T_k}\right) \leq 2 e^{-\bar{c}\eta_k^{-1}a^2}, \quad \forall a > 0,
\end{equation}
where the $\sigma$-algebra $\mathcal{F}_{T_{k}}$ is defined by $\mathcal{F}_{T_{k}} := \sigma(\bar{X}_s, \bar{Y}_s; s \leq T_k)$, and $\bar{c}$ is a positive constant independent of $t$, $k$, $\xi$  and $a$. Consequently, 
\begin{equation}
    \mathbb{P}\left(\sup_{T_k \leq s \leq t}|\zeta_s^{\tau_j}| > \bar{c}^{-\frac{1}{2}}\eta_k^{\frac{1}{2}}|\log \eta_k|^{\frac{1}{2}}\Big| \mathcal{F}_{T_k} \right) \leq 2 \eta_k  \rightarrow 0 \quad as \quad \eta_k \rightarrow 0.
\end{equation}
\end{lemma}

\begin{proof}
We prove the subgaussian property \eqref{eq:subguassianprop} via the well-known $\psi_2$-condition \cite{vershynin2018high}: there exists $\alpha > 0$ such that
\begin{equation}\label{eq:psi2}
    \mathbb{E}\left[e^{\alpha|\theta_t^{\tau_j}|^2} \Big| \mathcal{F}_{T_k}\right] \leq 2,
\end{equation}
where we denote $\theta_t^{\tau_j} := \sup_{T_k \leq s \leq t} \zeta_s^{\tau_j}$. Clearly, $\zeta_t^{\tau_j}$ of the form \eqref{eq:zetaj_def} is a martingale by optional stopping theorem \cite{durrett2018stochastic}, and its quadratic variation satisfies
$\langle \zeta_t^{\tau_j}\rangle\le t\wedge\tau_j-T_k\wedge\tau_j\le \eta_k$. Then it holds by the Burkholder-Davis-Gundy (BDG) inequality (\cite{barlow1982semi,carlen1991lp})  that
\begin{multline}\label{eq:taylorbdg}
    \mathbb{E}\left[e^{\alpha |\theta_t^{\tau_j}|^2}\Big| \mathcal{F}_{T_k}\right] = 1 + \sum_{p=1}^{+\infty} \frac{1}{p!} \alpha^p \mathbb{E}\left[|\theta_t^{\tau_j}|^{2p}\Big| \mathcal{F}_{T_k}\right]\\
    \leq 1 + \sum_{p=1}^{+\infty} \frac{1}{p!}\alpha^p C_{2p} \mathbb{E}\left[\langle \zeta^{\tau_j} \rangle_{T_{k+1}}^p\Big| \mathcal{F}_{T_k}\right] \leq 1 + \sum_{p=1}^{+\infty} \frac{1}{p!} C_{2p} \left(\eta_k \alpha\right)^p,
\end{multline}
where $\alpha$ is a positive parameter to be determined, and $C_{2p}$  is a positive constant satisfying \cite{barlow1982semi,carlen1991lp,pinelis1994optimum}:
\begin{equation}\label{claim:Cq}
    C_{2p} \le (C \sqrt{2p})^{2p},
\end{equation}
where $C$ is a positive constant related to the Hilbert space only (in our case $\mathbb{R}^d$).  Combining \eqref{eq:taylorbdg} and \eqref{claim:Cq}, we have
\begin{equation*}
    \mathbb{E}\left[e^{\alpha |\theta_t^{\tau_j}|^2}\Big| \mathcal{F}_{T_k}\right] \leq 1 + C\sum_{p=1}^{+\infty} \frac{p^p}{p!}\left(2\eta_k \alpha \right)^p.
\end{equation*}
Clearly, $\frac{p^p}{p!} \leq e^p p^{-\frac{1}{2}} \leq e^p$, which can be derived from an intermediate result in the proof of Stirling's formula \cite{rudin1976principles}: $\log p! > \left(p+\frac{1}{2}\right)\log p - p$. Therefore,
\begin{equation*}
    \mathbb{E}\left[e^{\alpha |\theta_t^{\tau_j}|^2}\Big| \mathcal{F}_{T_k}\right] \leq 1+ C \sum_{p=1}^{+\infty} \left(2e\eta_k \alpha\right)^p = 1 + C\frac{2e\eta_k \alpha}{1 - 2e\eta_k \alpha} =  2,
\end{equation*}
by choosing $\alpha = \frac{1}{2e(1+C)\eta_k } := \bar{c} \eta_k^{-1}$. Therefore, the $\psi_2$ condition \eqref{eq:psi2} holds.

Finally, using Chernoff's bound \cite{vershynin2018high}, for any $a>0$, it holds that
\begin{equation}\label{eq:chernoff}
    \mathbb{P}\left(|\theta_t^{\tau_j}| > a\Big| \mathcal{F}_{T_k}\right) \leq \mathbb{E}\left[e^{\alpha |\theta_t^{\tau_j}|^2}\Big| \mathcal{F}_{T_k}\right] / e^{\alpha a^2} \leq 2 e^{-\bar{c}\eta_k^{-1} a^2}.
\end{equation}
Consequently, taking $a = \bar{c}^{-\frac{1}{2}}\eta_k^{\frac{1}{2}}|\log \eta_k|^{\frac{1}{2}}$ gives the last claim.
\end{proof}

We will make use of the subgaussian estimate to control a series of conditional expectations. In particular, later we need the conditional expectations on events like $|Z_{T_k}^{\tau_j}| > R$
and $t<\tau_j$. If these two events are independent, there is little difficulty. The difficulty is that these two are highly correlated. Actually, we will make use of the fact that the former event almost is contained in the second one so that the estimates can carry through as well.

As a start, we prove the following conditional estimate of $|\zeta_t^{\tau_j}|$ as an illustration.
\begin{corollary}\label{lmm:abs_zeta}
Let $\eta_k \leq \min(\frac{1}{2K},  \frac{\bar{c} \beta R^2}{128\log 8})$, where $\bar{c}$ is the positive constant obtained in Lemma \ref{lmm:zeta}. Suppose that $\mathbb{P}(|Z_{T_k}^{\tau_j}| > R)>0$. Then, for $j$ large enough,  it holds that
    \begin{equation}\label{eq:lmmabs}
        \mathbb{E}\left[|\zeta_t^{\tau_j}| \Big| |Z_{T_k}^{\tau_j}| > R, t<\tau_j\right] \leq 8\sqrt{2} \sqrt{\eta_k}, \quad \forall t \in [T_k,T_{k+1}).
    \end{equation}
\end{corollary}
\begin{proof}
For simplicity, we denote the events
\begin{equation*}
    A := \left\{|Z_{T_k}^{\tau_j}|>R \right\}, \quad B:=\left\{ t<\tau_j\right\}.
\end{equation*}
Our goal is then to control $\mathbb{E}[|\zeta_t^{\tau_j}|\textbf{1}_A \textbf{1}_B]/P(A\cap B)$.

 Firstly, using the BDG inequality for $p\in (0, 2)$ (see, e.g., \cite[Theorem 7.3]{mao2007stochastic}), one has
\begin{gather}
    \mathbb{E}\left[|\zeta_t^{\tau_j}| \Big|  A \right] 
    \leq 4\sqrt{2} \mathbb{E}\left[\mathbb{E}\left[\langle\zeta^{\tau_j}\rangle_{T_k}^{\frac{1}{2}} \Big| \mathcal{F}_{T_k} \right]\Big| A \right] \leq 4\sqrt{2}\sqrt{\eta_k}.
\end{gather}

Next, we estimate $\mathbb{P}(B^c| A)$.
By Markov inequality  and the moment control in Lemma \ref{lmm:sgldmoment}, 
\begin{equation*}
    \mathbb{P}(|Z_{T_k}^{\tau_j}| \ge j_0) \leq \frac{\mathbb{E}|Z_{T_k}^{\tau_j}|^2}{j_0^2} \rightarrow 0 \quad \text{as} \quad j_0 \rightarrow \infty.
\end{equation*}
Hence, for $j_0$ large enough (independent of $j$), 
\begin{equation*}
    \mathbb{P}(|Z_{T_k}^{\tau_j}| \ge j_0) \le \frac{1}{4} \mathbb{P}(A).
\end{equation*}
Clearly,
\[
\mathbb{P}(B^c| A)\le \frac{\mathbb{P}(|Z_{T_k}^{\tau_j}| \ge j_0)}{\mathbb{P}(A)}+\frac{\mathbb{P}(B^c \cap\{R<|Z_{T_k}^{\tau_j}|<j_0\})}{\mathbb{P}(\{R<|Z_{T_k}^{\tau_j}|<j_0\})}.
\]
Now, since
\begin{equation*}
\begin{aligned}
 \left||Z_s^{\tau_j}| - |Z_{T_k}^{\tau_j}|\right| &\leq  \eta_k K |Z_{T_k}^{\tau_j}| + 2\sqrt{2\beta^{-1}}|\zeta_s^{\tau_j}|,
\end{aligned}
\end{equation*}
then $B^c$ could happen only if 
\[
2\sqrt{2\beta^{-1}}\sup_{T_k \le s\le t}|\zeta_s^{\tau_j}|
\ge \max\left\{j-\frac{3}{2}|Z_{T_k}^{\tau_j}|, \frac{1}{2}|Z_{T_k}^{\tau_j}|-j^{-1}\right\}.
\]
Taking $j$ with $j>\frac{3}{2}j_0+R/4$ and $j^{-1}<R/4$, then
\[
\frac{\mathbb{P}(B^c \cap\{R<|Z_{T_k}^{\tau_j}|<j_0\})}{\mathbb{P}(\{R<|Z_{T_k}^{\tau_j}|<j_0\})}
\le \mathbb{P}\left(2\sqrt{2\beta^{-1}}\sup_{T_k \le s\le t}|\zeta_s^{\tau_j}|>R/4 | \mathcal{F}_{T_k}\right).
\]
By Lemma \ref{lmm:zeta}, one obtains that 
\[
\mathbb{P}(2\sqrt{2\beta^{-1}}\sup_{T_k \le s\le t}|\zeta_s^{\tau_j}|>R/4 | \mathcal{F}_{T_k})\leq 2e^{-\bar{c}\beta R^2 \eta_k^{-1}/128}.
\]
Hence, for $\eta_k \leq \frac{\bar{c} \beta  R^2}{128\log 8}$, one has
\begin{equation}
    \mathbb{P}(B^C | A) < \frac{1}{2}.
\end{equation}
Consequently,
\begin{equation*}
     \frac{\mathbb{E}\left[|\zeta_t^{\tau_j}|\textbf{1}_A \textbf{1}_B\right]}{\mathbb{P}(A\cap B)} \leq  \frac{\mathbb{E}\left[|\zeta_t^{\tau_j}|\textbf{1}_A \right]}{\mathbb{P}(A) - \mathbb{P}(B^C \cap A)} 
     =\frac{\mathbb{E}\left[|\zeta_t^{\tau_j}| | A \right]}{1 - \mathbb{P}(B^C | A)} 
     \le 2\mathbb{E}\left[|\zeta_t^{\tau_j}| \Big|  A \right],
\end{equation*}
and the claim then follows.
\end{proof}

Next, we will make use the same idea to establish a series of conditional expectations, which is based on the tail estimate in Lemma \ref{lmm:zeta}.
\begin{lemma}\label{lmm:smallprobcontrol}
Let $\eta_k \in (0,1/2K)$. Then, for $j \ge j_0+R/4$, it holds that
\begin{equation}\label{eq:epsilonj}
    \mathbb{E}\left[|Z_{T_k}^{\tau_j}| \textbf{1}_{\{|Z_{T_k}^{\tau_j}| > R \}} \textbf{1}_{\{t \geq \tau_j \}}\right] \leq (\epsilon(j_0)+ 
 2e^{-\bar{c}\beta R^2 \eta_k^{-1}/128} )\mathbb{E}\left[|Z_{T_k}^{\tau_j}|\textbf{1}_{\{|Z_{T_k}^{\tau_j}| > R \}}\right],
\end{equation}
where $\epsilon(j_0)\to 0$ as $j_0\to\infty$. 
\end{lemma}

\begin{proof}
Clearly, \eqref{eq:epsilonj} is trivial if $\mathbb{P}(|Z_{T_k}^{\tau_j}| > R)=0$. Below, we assume that
$\mathbb{P}(|Z_{T_k}^{\tau_j}| > R)>0$.

On one hand, using the result for moment control (Lemma \ref{lmm:sgldmoment}), 
\begin{equation}\label{eq::338}
\mathbb{E}\left[|Z_{T_k}^{\tau_j}|\textbf{1}_{\{|Z_{T_k}^{\tau_j}| > j_0 \}}\textbf{1}_{\{t \geq \tau_j \}}\right] \leq \mathbb{E}\left[|Z_{T_k}^{\tau_j}|\textbf{1}_{\{|Z_{T_k}^{\tau_j}| > j_0 \}}\right] \leq \epsilon(j_0) \mathbb{E}\left[|Z_{T_k}^{\tau_j}|\textbf{1}_{\{|Z_{T_k}^{\tau_j}| > R \}}\right],
\end{equation}
where $\epsilon_0(j_0)\to 0$ uniformly in $j$ as $j_0\to\infty$.

Fix $j_0$ and let $|Z_{T_k}^{\tau_j}|=z\in (R, j_0]$. By similar discussion as in the proof of Corollary \ref{lmm:abs_zeta}, since $\eta_k K\le 1/2$, for $j\ge j_0+R/4$ and $j^{-1}\le R/4$, for $\tau_j\le t$, one necessarily needs
$2\sqrt{2\beta^{-1}}\sup_{T_k\le s\le t}|\zeta_t^{\tau_j}|\ge R/4$. Hence, for such $j$, one has
\begin{equation}
    \mathbb{P}\left(t \geq \tau_j | |Z_{T_k}^{\tau_j}|=z \right) \leq 2e^{-\bar{c}\beta 
 R^2 \eta_k^{-1}/ 128}.
\end{equation}
Hence, letting $\mu_{T_k}(dz)$ be the law of $|Z_{T_k}^{\tau_j}|$, one has
\begin{multline}\label{eq::340}
\mathbb{E}\left[|Z_{T_k}^{\tau_j}|\textbf{1}_{\{R<|Z_{T_k}^{\tau_j}| \leq j_0 \}}\textbf{1}_{\{t \geq \tau_j \}}\right] = \int_R^{j_0} z \mathbb{E}\left[ \textbf{1}_{\{t \geq \tau_j \}}\Big| |Z_{T_k}^{\tau_j}| = z \right] \mu_{T_k}(dz)\\
    \leq 2e^{-\bar{c}\beta R^2 \eta_k^{-1}/128} \int_R^{j_0}  z \mu_{T_k}(dz) \le 2e^{-\bar{c} \beta R^2 \eta_k^{-1}/128} \mathbb{E}\left[|Z_{T_k}^{\tau_j}|\textbf{1}_{\{|Z_{T_k}^{\tau_j}| > R \}}\right].
\end{multline}
Combining \eqref{eq::338} and \eqref{eq::340}, the claim \eqref{eq:epsilonj} holds.
\end{proof}

The following result follows from the same idea as above, but more involved. It tells us that the two random variables $|Z_t^{\tau_j}|$ and $|Z_{T_k}^{\tau_j}|$ are roughly the same. Note that this result is also based on the tail estimate in Lemma \ref{lmm:zeta}.
\begin{lemma}\label{lmm:lmm2}
Under Assumption \ref{ass}, for any $R_1 > 3R/2$ with $R$ obtained in Lemma \ref{lmm:ass}, and $\eta_k < \min\left(\frac{1}{2K}, \frac{1}{16}\beta\bar{c}R_1(R_1-3R/2), \beta\bar{c}R^2/(128 \log 5)\right)$, there exists $j_0 > R_1$ such that for all   $j > j_0$, 
\begin{equation}\label{eq:secondclaim}
    \mathbb{E}\left[|Z_{t}^{\tau_j}| \textbf{1}_{\{|Z_t^{\tau_j}| > R_1, |Z_{T_k}^{\tau_j}| \leq R, t<\tau_j\}}\right] \leq \varepsilon(\eta_k) \mathbb{E}\left[|Z_{t}^{\tau_j}| \textbf{1}_{\{ |Z_{T_k}^{\tau_j}| \leq R, t<\tau_j\}}\right],
\end{equation}
where
\begin{equation}
    \varepsilon(\eta) := \frac{15R_1 e^{-\bar{c}\beta(R_1 - 3R/2)^2 \eta^{-1}/ 8}}{R}.
\end{equation}
\end{lemma}

\begin{proof}

The idea is that for the event $\{|Z_t^{\tau_j}| > R_1, |Z_{T_k}^{\tau_j}| \leq R\}$ to happen, $|Z|$ must be $R$ for some time during $T_k$ and $t$. Conditioning on this, $|Z_t^{\tau_j}|$ should be large (roughly comparable to $R$), while the moment for $|Z_t^{\tau_j}|>R_1$ is then a small fraction of this conditional moment. 

Define the event
\[
E:=\{|Z_{T_k}^{\tau_j}| \leq R, \exists s\in [T_k, t], |Z_{s}^{\tau_j}|=R \}.
\]
If $\mathbb{P}(E)=0$, there is nothing to prove. Below, we assume that $\mathbb{P}(E)>0$. Again, the event $E$ is also almost contained in $\{t<\tau_j\}$.
We will in fact show that 
\begin{gather}\label{eq:epsilon_eta}
\begin{split}
&\quad \mathbb{E}\left[|Z_{t}^{\tau_j}| \textbf{1}_{\{|Z_t^{\tau_j}| > R_1, |Z_{T_k}^{\tau_j}| \leq R, t<\tau_j\}}\right]\\
&=\mathbb{E}\left[|Z_{t}^{\tau_j}| \textbf{1}_E \textbf{1}_{\{|Z_t^{\tau_j}| > R_1, t<\tau_j\}}\right]
\le \epsilon(\eta_k) \mathbb{E}\left[|Z_{t}^{\tau_j}| \textbf{1}_E \textbf{1}_{\{t<\tau_j\}}\right].
\end{split}
\end{gather}

For a nonnegative random variable $X$, one has
\begin{equation*}
\mathbb{E}\left[X\right]=\E\int_0^{\infty} \textbf{1}_{\{X> r\}} \,dr  = \int_0^{\infty} \mathbb{P}(X > r) dr.
\end{equation*}
Then,
\begin{multline*}
\mathbb{E}\left[|Z_{t}^{\tau_j}| \textbf{1}_E \textbf{1}_{\{|Z_t^{\tau_j}| > R_1, t<\tau_j\}}\right]
 =\int_0^{\infty}\mathbb{P}(|Z_{t}^{\tau_j}| \textbf{1}_E \textbf{1}_{\{|Z_t^{\tau_j}| > R_1, t<\tau_j\}}> r)\,dr\\
=R_1 \mathbb{P}\left(|Z_t^{\tau_j}| > R_1 , E, t < \tau_j \right) + \int_{R_1}^{\infty} \mathbb{P}\left(|Z_t^{\tau_j}| > r , E, t < \tau_j \right) dr.
\end{multline*}
By applying the strong Markov property for $|Z_s^{\tau_j}|$ hitting $R$ and using the same idea in the proof of Corollary \ref{lmm:abs_zeta}, one has for $j$ large enough that
\[
\mathbb{P}\left(|Z_t^{\tau_j}| > r , E, t < \tau_j \right)\le 2 \mathbb{P}(|Z_t^{\tau_j}| > r | E)
\mathbb{P}(E, t<\tau_j).
\]
Using the definition of the event $E$ and applying Lemma \ref{lmm:zeta} with the strong Markov property for $|Z_s^{\tau_j}|=R$, one has
\begin{equation}\label{eq:first_estimate}
    P\left(|Z_t^{\tau_j}| >r \Big| E \right) \leq 2e^{-\bar{c} \beta(r - 3R/2)^2 \eta_k^{-1} /8}, \quad \forall r \geq R_1 > 3R/2.
\end{equation}
Hence, one then has
\[
\begin{split}
&\mathbb{E}\left[|Z_{t}^{\tau_j}| \textbf{1}_E \textbf{1}_{\{|Z_t^{\tau_j}| > R_1, t<\tau_j\}}\right]\\
 & \le 2\left[R_1 e^{-\bar{c} \beta(R_1 - 3R/2)^2 \eta_k^{-1} /8} + \int_{R_1}^{\infty} e^{-\bar{c}\beta(r - 3R/2)^2 \eta_k^{-1} / 8} dr\right]
\mathbb{P}(E, t<\tau_j)\\
& \le 2\left(R_1 + 8\eta_k\bar{c}^{-1}\beta^{-1}(R_1-3R/2)^{-1}\right) e^{-\bar{c}\beta(R_1 - 3R/2)^2 \eta_k^{-1}/8}\mathbb{P}(E, t<\tau_j)\\
& \le 3R_1 e^{-\bar{c}\beta(R_1 - 3R/2)^2 \eta_k^{-1} / 8}\mathbb{P}(E, t<\tau_j).
\end{split}
\]

Next, we aim to show that \[
\mathbb{E}\left[|Z_{t}^{\tau_j}| \textbf{1}_E \textbf{1}_{\{t<\tau_j\}}\right]
\ge \frac{R}{5}\mathbb{P}(E, t<\tau_j).
\]
In fact, 
\[
\mathbb{E}\left[|Z_{t}^{\tau_j}| \textbf{1}_E \textbf{1}_{\{t<\tau_j\}}\right]
=\int_0^{\infty} \mathbb{P}\left(|Z_t^{\tau_j}| > r, E, t < \tau_j \right) dr
\ge \int_0^{R/4}\mathbb{P}\left(|Z_t^{\tau_j}| > r, E, t < \tau_j \right) dr.
\]
For $r\in [0, R/4]$,
\[
\mathbb{P}\left(|Z_t^{\tau_j}| > r, E, t < \tau_j \right)
\ge \mathbb{P}(E, t<\tau_j)-\mathbb{P}\left(|Z_t^{\tau_j}| \le R/4, E, t < \tau_j \right)
\]
Suppose $s$ is the stopping time for $|Z^{\tau_j}|$ hitting $R$ during $[T_k, t]$, then one needs
\[
2\sqrt{2\beta^{-1}}\left|\int_{s\wedge \tau_j}^{t \wedge \tau_j} \frac{Z_{t'}^{\otimes 2}}{|Z_{t'}|^2} \cdot dW_{t'}\right|\ge R/4,
\]
for $|Z_t^{\tau_j}|\le R/4$ and $t<\tau_j$ to happen (if $j$ is large enough).

By strong Markov property and similar estimate as in 
the proof of Corollary \ref{lmm:abs_zeta}, one has
\begin{equation}
\begin{split}
&\mathbb{P}\left(|Z_t^{\tau_j}| \le R/4, E, t < \tau_j \right)\le 2\mathbb{P}\left(|Z_t^{\tau_j}| \leq \frac{R}{4} \Big| E\right) \mathbb{P}(E, t<\tau_j)\\
&\le  4 e^{-\bar{c}\beta\frac{R^2}{128}\eta_k^{-1}}\mathbb{P}(E, t<\tau_j).
\end{split}
\end{equation}
Hence,
\[
\mathbb{P}\left(|Z_t^{\tau_j}| > r, E, t < \tau_j \right)\ge \frac{R}{5}\mathbb{P}(E, t<\tau_j),
\]
and the claim \eqref{eq:epsilon_eta} holds.

Note that the event $E \cap \{t < \tau_j \}$ is smaller than $\{|Z_{T_k}^{\tau_j}| \leq R \} \cap \{t < \tau_j \}$ so that
\begin{equation}\label{eq:epsilon_eta2}
    \mathbb{E}\left[|Z_{t}^{\tau_j}| \textbf{1}_E \textbf{1}_{\{t<\tau_j\}}\right] \leq \mathbb{E}\left[|Z_{t}^{\tau_j}| \textbf{1}_{\{|Z_{T_k}^{\tau_j}| \leq R \}} \textbf{1}_{\{t<\tau_j\}}\right].
\end{equation}
With \eqref{eq:epsilon_eta} in hand, the claim then follows.
\end{proof}

Next, we obtain the following estimate for the term $I_1(t)$ defined in \eqref{eq:aftersplit}, which explains how we treat the random batch at discrete time $T_k$ and make use of the far away convexity. Note that this result is also based on the tail estimate in Lemma \ref{lmm:zeta}.

Before the detailed derivation, we first give a brief summary of proof for Proposition \ref{pro:rbandconv} regarding the estimate for the random batch. After It\^o's calculation, one needs to estimate 
$$
\mathbb{E}\left[\phi(\bar{X}_t,\bar{Y}_t)(\nabla U^{\xi_k}(\bar{X}_{T_k}) - \nabla U^{\xi_k}(\bar{Y}_{T_k}))\textbf{1}_A\right]
$$
for some function $\phi(\cdot,\cdot)$, $t \in [T_k,T_{k+1})$ and some event $A$ independent of the random batch $\xi_k$. The key step in Proposition \ref{pro:rbandconv} is to use Taylor's expansion and consistency of the random batch:
\begin{equation*}
    \begin{aligned}
        &\quad\mathbb{E}\left[\phi(\bar{X}_t,\bar{Y}_t)(\nabla U^{\xi_k}(\bar{X}_{T_k}) - \nabla U^{\xi_k}(\bar{Y}_{T_k}))\textbf{1}_A\right]\\
        &= \mathbb{E}\left[\phi(\bar{X}_{T_k},\bar{Y}_{T_k})(\nabla U^{\xi_k}(\bar{X}_{T_k}) - \nabla U^{\xi_k}(\bar{Y}_{T_k}))\textbf{1}_A\right] + \epsilon(\eta)\\
        &= \mathbb{E}\left[\phi(\bar{X}_{T_k},\bar{Y}_{T_k})(\nabla U(\bar{X}_{T_k}) - \nabla U(\bar{Y}_{T_k}))\textbf{1}_A\right] + \epsilon(\eta),
    \end{aligned}
\end{equation*}
where the last equality is due to $\mathbb{E}_{\xi}[U^{\xi}(\cdot)]=U(\cdot)$ and the fact that $\xi_k$ is independent of $\bar{X}_{T_k}$ and $\bar{Y}_{T_k}$. Moreover, under the evene $A$, the small remainder term $\epsilon(\eta)$ can be estimated through the tail behavior obtained in Lemma \ref{lmm:zeta}. The details are given as follows.

\begin{proposition}\label{pro:rbandconv}
Let $f$ be the Lyapunov function defined in \eqref{eq:f_sgld}. Suppose Assumption \ref{ass} holds. Assume that $\eta_k \leq 1 / 2K$, $KR\eta_k \le \bar{c}^{-\frac{1}{2}}\eta_k^{\frac{1}{2}}|\log \eta_k|^{\frac{1}{2}} \leq 1$ ($\bar{c}$ is the constant coming from Lemma \ref{lmm:zeta}). Denote
\begin{gather}
c' := \bar{c}^{-1/2}\left(\frac{2K}{R/2-1}+K  c_f e^{-c_f(R/2-1)}\right) + \frac{4\bar{c}^{-1/2}}{R}.
\end{gather}
Then for $j$ sufficiently large, it holds that
\begin{multline}\label{eq:lmm1consequence}
    \mathbb{E}\left[-f'(|Z^{\tau_j}_t|)\frac{Z_t^{\tau_j}}{|Z_t^{\tau_j}|} \cdot \left(\nabla U^{\xi_k}(\bar{X}_{T_k}) - \nabla U^{\xi_k}(\bar{Y}_{T_k}) \right) \textbf{1}_{\{|Z_{T_k}^{\tau_j}| > R \}} \textbf{1}_{\{t < \tau_j \}}\right]\\
    \leq -\left(e^{-c_f R_1}\kappa - c' \eta_k^{\frac{1}{2}}|\log \eta_k|^{\frac{1}{2}}-3e^{-\bar{c}\beta R^2 \eta_k^{-1}/128}\right)\mathbb{E}\left[|Z_{T_k}^{\tau_j}|\textbf{1}_{\{|Z_{T_k}^{\tau_j}| > R \}}\right].
\end{multline}
\end{proposition}
\begin{proof}

We first show that for $j\in \mathbb{N}_{+}$, $t \in [T_k,T_{k+1})$, 
\begin{multline}\label{eq:eq__333}
     \mathbb{E}\left[-f'(|Z^{\tau_j}_t|)\frac{Z_t^{\tau_j}}{|Z_t^{\tau_j}|} \cdot \left(\nabla U^{\xi_k}(\bar{X}_{T_k}) - \nabla U^{\xi_k}(\bar{Y}_{T_k}) \right)  \Big| \mathcal{F}_{T_k}\right]\\
    \leq \left(-f'(|Z_{T_k}^{\tau_j}|)\frac{Z_{T_k}^{\tau_j}}{|Z_{T_k}^{\tau_j}|} \cdot \left(\nabla U^{\xi_k}(\bar{X}_{T_k}) - \nabla U^{\xi_k}(\bar{Y}_{T_k}) \right) + c' \eta_k^{\frac{1}{2}}|\log \eta_k|^{\frac{1}{2}} |Z_{T_k}^{\tau_j}| \right)\textbf{1}_{\{|Z_{T_k}^{\tau_j}| > R \}}.
\end{multline}
Recall that for $t \in [T_k,T_{k+1})$, 
\begin{equation}
    Z_t^{\tau_j} = Z_{T_k}^{\tau_j} - (t\wedge \tau_j - T_k\wedge \tau_j) A_{T_k} + 2\sqrt{2\beta^{-1}}\zeta_t^{\tau_j}.
\end{equation}
Noting that $|f'(r)|\le 1$, by Assumption \ref{ass} and Lemma \ref{lmm:zeta}, it follows easily that
\begin{multline}\label{eq_3_6}
 \left| \mathbb{E}\left[-f'(|Z_t^{\tau_j}|)\frac{Z_t^{\tau_j}}{|Z_t^{\tau_j}|} \cdot A_{T_k} \textbf{1}_{\{|Z_{T_k}^{\tau_j}| > R \}} \textbf{1}_{\{|\zeta_t^{\tau_j}| > \bar{c}^{-\frac{1}{2}} \eta_k^{\frac{1}{2}}|\log \eta_k|^{\frac{1}{2}} \}} \Big| \mathcal{F}_{T_k}\right] \right| \\
 \leq  2\eta_k K|Z_{T_k}^{\tau_j}|\textbf{1}_{\{|Z_{T_k}^{\tau_j}| > R \}}.
\end{multline}
On the other hand, consider the following
\[
\mathbb{E}\left[-f'(|Z_t^{\tau_j}|)\frac{Z_t^{\tau_j}}{|Z_t^{\tau_j}|} \cdot A_{T_k} \textbf{1}_{\{|Z_{T_k}^{\tau_j}| > R \}} \textbf{1}_{\{|\zeta_t^{\tau_j}| \leq \bar{c}^{-\frac{1}{2}} \eta_k^{\frac{1}{2}}|\log \eta_k|^{\frac{1}{2}} \}} \Big| \mathcal{F}_{T_k}\right].
\]
Consider the function $g$: $\mathbb{R}^d \rightarrow \mathbb{R}$ defined by   \[
g(x) := -f'(|x|)\frac{x}{|x|} \cdot A_{T_k}.
\]
 Here, $A_{T_k} =  \nabla U^{\xi_k} (\bar{X}_{T_k}) - \nabla U^{\xi_k}(\bar{Y}_{T_k}) $ is $\mathcal{F}_{T_k}$ measurable, satisfying $|A_{T_k}| \leq K|Z_{T_k}|$ by Assumption \ref{ass}. Clearly, for $x \neq 0$, the gradient $\nabla g(x) = -f'(|x|)A_{T_k} \cdot \frac{1}{|x|} \left( I_d - \frac{x^{\otimes 2}}{|x|^2} \right) - f''(|x|) \frac{x^{\otimes 2}}{|x|^2} \cdot A_{T_k}$ is well-defined.  Hence
 \[
\nabla g(\lambda Z_{T_k}^{\tau_j} + (1 -\lambda) Z_t^{\tau_j})\cdot(-A_{T_k})
\le \frac{|A_{T_k}|^2}{|\lambda Z_{T_k} ^{\tau_j}+ (1 -\lambda) Z_t^{\tau_j}|}\le \frac{K^2|Z_{T_k} ^{\tau_j}|^2}{|\lambda Z_{T_k} ^{\tau_j}+ (1 -\lambda) Z_t^{\tau_j}|},
 \]
 where we used $f''(r)A_{T_k}\cdot \frac{x^{\otimes 2}}{|x|^2}\cdot A_{T_k}\le 0$. Then, for $|\zeta_t^{\tau_j}| \leq \bar{c}^{-\frac{1}{2}} \eta_k^{\frac{1}{2}}|\log \eta_k|^{\frac{1}{2}}$,
\begin{equation*}
\begin{aligned}
    g(Z_t^{\tau_j}) 
    &= g(Z_{T_k}^{\tau_j}) + \left(\int_0^1 \nabla g(\lambda Z_{T_k}^{\tau_j} + (1 -\lambda) Z_t^{\tau_j})d\lambda \right) \cdot \left(Z_t^{\tau_j}-Z_{T_k}^{\tau_j} \right)\\
    & \leq g(Z_{T_k}^{\tau_j}) +\eta_k K^2\int_0^1 \frac{|Z_{T_k}^{\tau_j}|^2}{|\lambda Z_{T_k}^{\tau_j}+ (1 -\lambda) Z_t^{\tau_j}|}\,d\lambda\\
    &+\left(\int_0^1\left| \nabla g(\lambda Z_{T_k}^{\tau_j} + (1 -\lambda) Z_t^{\tau_j})\right|d\lambda\right) \bar{c}^{-\frac{1}{2}}\eta_k^{\frac{1}{2}}|\log \eta_k|^{\frac{1}{2}}.
\end{aligned}
\end{equation*}
Hence, choosing small $\eta_k$ such that  $K\eta_k \leq \frac{1}{2}$ and $\bar{c}^{-\frac{1}{2}}\eta_k^{\frac{1}{2}}|\log \eta_k|^{\frac{1}{2}} \leq 1$,  it is not difficult to control
$|\lambda Z_{T_k}^{\tau_j}+ (1 -\lambda) Z_t^{\tau_j}|
\ge |Z_{T_k}^{\tau_j}|(1-\eta_k K)-1\ge |Z_{T_k}^{\tau_j}|(1/2-1/R)$. Moreover, 
\begin{equation*}
    \begin{aligned}
        \left| \nabla g(\lambda Z_{T_k}^{\tau_j} + (1 -\lambda) Z_t^{\tau_j})\right| &\leq \frac{|A_{T_k}|}{|\lambda Z_{T_k} ^{\tau_j}+ (1 -\lambda) Z_t^{\tau_j}|} + c_f |A_{T_k}|\,e^{-c_f |\lambda Z_{T_k}^{\tau_j} + (1 -\lambda) Z_t^{\tau_j}|}\\
        & \leq \left(\frac{K}{R/2-1}+K c_f e^{-c_f(R/2-1)}\right)
        |Z_{T_k}^{\tau_j}|.
    \end{aligned}
\end{equation*}
 Hence,  for $KR\eta_k \le \bar{c}^{-\frac{1}{2}}\eta_k^{\frac{1}{2}}|\log \eta_k|^{\frac{1}{2}}$, 
\begin{equation*}
    \begin{aligned}
        g(Z_t^{\tau_j})\textbf{1}_{\{|Z_{T_k}^{\tau_j}| > R \}}
        &\leq  \left(g(Z_{T_k}^{\tau_j}) +  \bar{c}' \eta_k^{\frac{1}{2}}|\log \eta_k|^{\frac{1}{2}}|Z_{T_k}^{\tau_j}|\right)\textbf{1}_{\{|Z_{T_k}^{\tau_j}| > R \}},
    \end{aligned}
\end{equation*}
where
\begin{equation*}
    \bar{c}' := \bar{c}^{-1/2}\left(\frac{2K}{R/2-1}+K  c_f e^{-c_f(R/2-1)}\right) .
\end{equation*}
By Lemma \ref{lmm:zeta}, this then implies
\begin{multline}\label{eq_3_9}
    \mathbb{E}\left[-f'(|Z_t^{\tau_j}|)\frac{Z_t^{\tau_j}}{|Z_t^{\tau_j}|} \cdot A_{T_k} \textbf{1}_{\{|Z_{T_k}^{\tau_j}| > R \}} \textbf{1}_{\{|\zeta_t^{\tau_j}| \leq \bar{c}^{-1/2}\eta_k^{\frac{1}{2}}|\log \eta_k|^{\frac{1}{2}} \}} \Big| \mathcal{F}_{T_k}\right]\\
    \leq (1-2\eta_k)\left(-f'(|Z_{T_k}^{\tau_j}|)\frac{Z_{T_k}^{\tau_j}}{|Z_{T_k}^{\tau_j}|} \cdot A_{T_k} + \bar{c}' \eta_k^{\frac{1}{2}}|\log \eta_k|^{\frac{1}{2}} |Z_{T_k}^{\tau_j}| \right)\textbf{1}_{\{|Z_{T_k}^{\tau_j}| > R \}}\\
    \le \left(-f'(|Z_{T_k}^{\tau_j}|)\frac{Z_{T_k}^{\tau_j}}{|Z_{T_k}^{\tau_j}|} \cdot A_{T_k}+2\eta_k K |Z_{T_k}^{\tau_j}|+ \bar{c}' \eta_k^{\frac{1}{2}}|\log \eta_k|^{\frac{1}{2}} |Z_{T_k}^{\tau_j}| \right)\textbf{1}_{\{|Z_{T_k}^{\tau_j}| > R \}}.
\end{multline}

Combining \eqref{eq_3_6} and \eqref{eq_3_9}, the claim \eqref{eq:eq__333} holds.

Next, we prove \eqref{eq:lmm1consequence}. Clearly,
\begin{equation*}
\begin{aligned}
    &\quad\mathbb{E}\left[-f'(|Z^{\tau_j}_t|)\frac{Z_t^{\tau_j}}{|Z_t^{\tau_j}|} \cdot \left(\nabla U^{\xi_k}(\bar{X}_{T_k}) - \nabla U^{\xi_k}(\bar{Y}_{T_k}) \right) \textbf{1}_{\{|Z_{T_k}^{\tau_j}| > R \}} \textbf{1}_{\{t < \tau_j \}}\right]\\
    &=\mathbb{E}\left[-f'(|Z^{\tau_j}_t|)\frac{Z_t^{\tau_j}}{|Z_t^{\tau_j}|} \cdot \left(\nabla U^{\xi_k}(\bar{X}_{T_k}) - \nabla U^{\xi_k}(\bar{Y}_{T_k}) \right) \textbf{1}_{\{|Z_{T_k}^{\tau_j}| > R \}} \right]\\
    &\quad - \mathbb{E}\left[-f'(|Z^{\tau_j}_t|)\frac{Z_t^{\tau_j}}{|Z_t^{\tau_j}|} \cdot \left(\nabla U^{\xi_k}(\bar{X}_{T_k}) - \nabla U^{\xi_k}(\bar{Y}_{T_k}) \right) \textbf{1}_{\{|Z_{T_k}^{\tau_j}| > R \}} \textbf{1}_{\{t \geq \tau_j \}}\right]=B_1+B_2.
\end{aligned}
\end{equation*}
For the first term, using \eqref{eq:eq__333}, the consistency of the random batch $\xi$, and the convexity condition in Assumption \ref{ass}, one has
\begin{equation}\label{eq:firstterm}
\begin{aligned}
    B_1 
    &\leq \mathbb{E}\left[ \left(-f'(|Z_{T_k}^{\tau_j}|)\frac{Z_{T_k}^{\tau_j}}{|Z_{T_k}^{\tau_j}|} \cdot \left(\nabla U^{\xi_k}(\bar{X}_{T_k}) - \nabla U^{\xi_k}(\bar{Y}_{T_k}) \right) + c' \eta_k^{\frac{1}{2}}|\log \eta_k|^{\frac{1}{2}} |Z_{T_k}^{\tau_j}| \right)\textbf{1}_{\{|Z_{T_k}^{\tau_j}| > R \}}\right]\\
    &= \mathbb{E}\left[ \left(-f'(|Z_{T_k}^{\tau_j}|)\frac{Z_{T_k}^{\tau_j}}{|Z_{T_k}^{\tau_j}|} \cdot \left(\nabla U(\bar{X}_{T_k}) - \nabla U(\bar{Y}_{T_k}) \right) + c' \eta_k^{\frac{1}{2}}|\log \eta_k|^{\frac{1}{2}} |Z_{T_k}^{\tau_j}| \right)\textbf{1}_{\{|Z_{T_k}^{\tau_j}| > R \}}\right]\\
    &\leq -\left(e^{-c_f R_1}\kappa - c' \eta_k^{\frac{1}{2}}|\log \eta_k|^{\frac{1}{2}}\right)\mathbb{E}\left[|Z_{T_k}^{\tau_j}|\textbf{1}_{\{|Z_{T_k}^{\tau_j}| > R \}}\right],
\end{aligned}
\end{equation}
where we used the convexity outside $B(0, R)$ in the last inequality.

For the second term, by Lipschitz condition in Assumption \ref{ass}, we have
\begin{gather}\label{eq:secondterm}
    |B_2|
    \leq K\mathbb{E}\left[|Z_{T_k}^{\tau_j}| \textbf{1}_{\{|Z_{T_k}^{\tau_j}| > R \}} \textbf{1}_{\{t \geq \tau_j \}}\right].
\end{gather}

Finally, combining \eqref{eq:firstterm}, \eqref{eq:secondterm} and \eqref{eq:epsilonj} in Lemma \ref{lmm:smallprobcontrol}, we conclude that for $j$ large enough
\begin{multline}
    \mathbb{E}\left[-f'(|Z^{\tau_j}_t|)\frac{Z_t^{\tau_j}}{|Z_t^{\tau_j}|} \cdot \left(\nabla U^{\xi_k}(\bar{X}_{T_k}) - \nabla U^{\xi_k}(\bar{Y}_{T_k}) \right) \textbf{1}_{\{|Z_{T_k}^{\tau_j}| > R \}} \textbf{1}_{\{t < \tau_j \}}\right]\\
    \leq -\left(e^{-c_f R_1}\kappa - c' \eta_k^{\frac{1}{2}}|\log \eta_k|^{\frac{1}{2}} -3e^{-\bar{c}\beta R^2 \eta_k^{-1}/128}\right)\mathbb{E}\left[|Z_{T_k}^{\tau_j}|\textbf{1}_{\{|Z_{T_k}^{\tau_j}| > R \}}\right].
\end{multline}
\end{proof}

\begin{lemma}\label{lmm:tauj}
It holds that $\tau_j$ is nondecreasing in $j$ and 
\begin{equation}
    \tau_j \to \tau ,\quad a.s.,
\end{equation}
where $\tau$ is a stopping time defined in \eqref{eq:deftau}.
\end{lemma}

\begin{proof}[Proof of Lemma \ref{lmm:tauj}:]
It is clear that $\tau_j$ is nondecreasing in $j$ and $\sup_j \tau_j \le \tau$.
Fix $T>0$.
By H\"older's inequality and  Lemma \ref{lmm:sgldmoment}, for any $T>0$,
\begin{equation}\label{eq:Zcontrol}
\begin{split}
    \mathbb{E}\left[\sup_{s\leq t}|Z_s|\right] &\leq \left(\mathbb{E}\left[\sup_{s\leq t}|Z_s|^2\right]\right)^{\frac{1}{2}} \\
    &\leq \left( 2\mathbb{E}\left[\sup_{s\leq t}|\bar{X}_s|^2 \right]+ 2\mathbb{E}\left[\sup_{s\leq t}|\bar{Y}_s|^2 \right]\right)^{\frac{1}{2}} \leq C(T),
    \end{split}
    \quad \forall t \in [0,T],
\end{equation}
where $C(T)$ is a positive constant. 
Then, $\lim_{j\to\infty}\mathbb{P}(|Z_{\tau_j \wedge T}|\ge j) \leq \lim_{j \rightarrow \infty} \frac{C(T)}{j}=0$. By continuity and the definition of $\tau_j$, $\tau$, one has
\[
\sup_j \tau_j \wedge T=\tau \wedge T
\]
Since $T$ is arbitrary, $\tau_j \to \tau$.
\end{proof}

\begin{remark}
Lemma \ref{lmm:tauj} can also be proved based on moment control for the stopped process $\bar{X}_s^{\tau_j}=\bar{X}_{s\wedge\tau_j}$, which is a weaker result compared with Lemma \ref{lmm:sgldmoment}, and is in fact easier to prove than the moment control of $\sup_{s\le t}|\bar{X}_s|$ in Lemma \ref{lmm:sgldmoment}.
\end{remark}


\begin{remark}[Discussion on dimension dependency for the contraction rate]\label{rmk:d}
We believe that the contraction rate $c$ in our result is dimension-free. The only place that might be influenced by the dimension $d$ is the positive constant $C_{2p}$ in \eqref{claim:Cq} coming from the BDG inequality. So far, we have not found any reference claiming that the constant in the BDG inequality is independent of the dimension $d$. However, the process $\zeta_t$ defined in \eqref{eq:zeta_def} resembles a 1D Brownian Motion, since the rank of the matrix $Z_t^{\otimes 2} / |Z_t|^2$ is one with trace to be $1$. If the direction $Z_t/|Z_t|$ does not change, then it is exactly 1D Brownian motion. The difference is that the direction is changing along time.
\end{remark}

\section{General drift case}\label{sec:general}
In many applications, the drift term may not be of the form $-\nabla U$. In this section, we generalize the results to the general diffusion processes where the drifts are no longer gradients, namely, one still has ergodicity for the random batch version of the Euler-Maruyama scheme for diffusion processes. 

Consider the time continuous diffusion process
\begin{equation}\label{eq:overdamped}
    dX = b(X)\, dt + \sqrt{2\beta^{-1}}dW,
\end{equation}
where $b(\cdot)$ is a given drift, which might not be a gradient field (in this case, $X$ is no longer a Langevin diffusion).
Let $b^{\xi}$ be an unbiased stochastic estimate of $b$, $\E(b^{\xi}(\cdot))=b(\cdot)$. The random batch version of the Euler-Maruyama scheme for the continuous SDE \eqref{eq:overdamped}, the correspondence of SGLD iteration, is then given by
\begin{equation}\label{eq:randomem}
    \bar{X}_{T_{k+1}} = \bar{X}_{T_k} + \eta_k b^{\xi_k}(\bar{X}_{T_k}) + \sqrt{2\beta^{-1}}(W_{T_{k+1}} - W_{T_k}).
\end{equation}
Similarly with \eqref{eq:sgld}, we only need to analyze the following time interpolation:
\begin{equation}\label{eq:randomemcontinuous}
    \bar{X}_t = \bar{X}_{T_k} - \int_{T_k}^t b^{\xi_k} (\bar{X}_{T_k}) ds + \int_{T_k}^t\sqrt{2\beta^{-1}} dW_s,  \quad t \in [T_k, T_{k+1}),\quad k = 0,1,\dots.
\end{equation}

In order to obtain similar contraction property, we need the following assumption, which corresponds with Assumption \ref{ass}.
\begin{assumption}\label{ass:general}
There exist $R >0$, $\kappa > 0$, $K > 0$ such that the followings hold:
    \begin{itemize}
        \item[(a)]$\forall |x-y| > R$,
        \begin{equation}
            -(x-y) \cdot (b(x) - b(y)) \geq \kappa |x-y|^2;
        \end{equation}
        \item[(b)]$\forall x,y \in \mathbb{R}^d$, $\forall \xi \in \mathcal{S}$,
        \begin{equation}
            \left|b^{\xi}(x) - b^{\xi}(y) \right| \leq K |x-y|.
        \end{equation}
    \end{itemize}
\end{assumption}
Again, the first assumption can be obtained if for some $\kappa_0>0$ and $R_0>0$ the following holds 
\begin{gather}
-v\cdot \nabla b(x)\cdot v \ge \kappa_0|v|^2, \forall x\in \mathbb{R}^d\setminus B(0, R_0), \forall v\in \mathbb{R}^d.
\end{gather}

Then, we are able to prove the contraction property for the algorithm \eqref{eq:randomem}, which naturally implies geometric ergodicity for constant step size $\eta_k\equiv \eta$ by contraction mapping theorem  \cite{lax2002functional}.

\begin{theorem}\label{thm:contractiongeneral}[Wasserstein contraction with general drift]
Consider the random Euler-Maruyama iteration \eqref{eq:randomemcontinuous}. For any two initial distributions $\mu_0$ and $\nu_0$, denote $\mu_{T_k}$ and $\nu_{T_k}$ to be the corresponding laws at $T_k$. Denote $h := \sup_k \eta_k$. Let $f$ be the Lyapunov function defined in \eqref{eq:f_sgld}. Then under Assumption \ref{ass:general}, for fixed $R_1 = 2R$ and $c_f$ satisfying $\frac{1}{3}\sqrt{2\beta^{-1}} c_f R^{-1} - K \ge 0$, there exists  $\delta>0$ such that for $h\le \delta$, the following Wasserstein contraction result holds:
\begin{equation}
    W_f(\mu_{T_k}, \nu_{T_k}) \leq e^{-cT_k}W_f(\mu_0,\nu_0), \quad k \in \mathbb{N},
\end{equation}
where
\begin{equation*}
c = \frac{1}{3}e^{-2c_f R}\min\left(\sqrt{2\beta^{-1}} c_f R^{-1}/2, \kappa \right).
\end{equation*}
Consequently,
\begin{equation}
    W_1(\mu_{T_k}, \nu_{T_k}) \leq c_0 e^{-cT_k}W_1(\mu_0,\nu_0), \quad k \in \mathbb{N},\quad c_0 := e^{2c_f R}.
\end{equation}
Moreover, if $\eta_k \equiv \eta$ is a constant, then for $\eta\le \delta$, the iteration \eqref{eq:randomem} has a unique invariant distribution $\tilde{\pi}$ such that
\begin{equation}
    W_1(\mu_{T_k}, \tilde{\pi}) \leq c_0 e^{-cT_k}W_1(\mu_0,\tilde{\pi}), \quad k \in \mathbb{N}.
\end{equation}
\end{theorem}

The proof is almost the same as that of Theorem \ref{thm:contraction}. The only difference is that we use condition (a) in Assumption \ref{ass:general} instead of the convexity condition (condition (a) in Assumption \ref{ass}) when estimating the term $I_3$ near \eqref{eq:I3estimate}.

\section{Conclusion}
In this paper, We proved the geometric ergodicity of the Stochastic Gradient Langevin Dynamics (SGLD) algorithm under nonconvexity settings. As a popular online sampling algorithm, SGLD has shown exceptional performance when dealing with high-dimensional and large-scaled data. Via the technique of reflection coupling, in Theorem \ref{thm:contraction} we proved the Wasserstein contraction of SGLD when the target distribution is log-concave only outside some compact set. In particular, the time discretization and the minibatch in SGLD introduced several difficulties when applying the reflection coupling, which were addressed by a series of careful estimates of conditional expectations. As a direct corollary, we proved that the SGLD with constant step size has an invariant distribution and obtained its geometric ergodicity in terms of $W_1$ distance. The generalization to non-gradient drifts was also included. We also remarked that the contraction rate $c$ in Theorem \ref{thm:contraction} is intuitively dimension-free, as discussed in Remark \ref{rmk:d}. Remarkably, we believe that many techniques in this article are applicable to other discrete algorithms involving random batches.

\section*{Acknowledgements}
This work is financially supported by the National Key R\&D Program of China, Project Number 2021YFA1002800 and Number 2020YFA0712000. The work of L. Li was partially supported by NSFC 12371400 and 12031013,  Shanghai Municipal Science and Technology Major Project 2021SHZDZX0102, Shanghai Science and Technology Commission (Grant No. 21JC1403700, 21JC1402900). The work of J.-G. Liu was partially supported by NSF DMS-2106988.

\appendix

\section{Missing Proofs}\label{sec:lmmmoment}
In this section, we give the detailed proofs for Lemma \ref{lmm:ass} and Lemma \ref{lmm:sgldmoment}.

\begin{proof}[Proof of Lemma \ref{lmm:ass}]
Fix $r>2R_0$, and choose two arbitrary point $x,y \in \mathbb{R}^d$ satisfying $|x-y| = r$. Then it holds that
\begin{equation}\label{eq:1:5}
    \frac{(x-y) \cdot (\nabla U(x) - \nabla U(y))}{|x-y|^2} = \frac{(x-y) \cdot \int_0^1 \nabla^2 U\left(tx + (1-t) y\right)dt \cdot (x-y)}{|x-y|^2}.
\end{equation}
For fixed $x,y$ above, denote
\begin{equation*}
    A_1 := \left\{t\in(0,1): |tx+(1-t)y| \leq R_0 \right\},
\end{equation*}
and
\begin{equation*}
    A_2 := \left\{t\in(0,1): |tx+(1-t)y| > R_0 \right\}.
\end{equation*}
Then, by Assumption \ref{ass},
\begin{equation}\label{eq:1:6}
\begin{aligned}
    \int_0^1 \nabla^2 U\left(tx + (1-t) y\right)dt &= \left(\int_{A_1} + \int_{A_2}\right) \nabla^2 U\left(tx + (1-t) y\right)dt\\
    &\succeq \left(-m(A_1)K I_d \right) + \left(m(A_2)\kappa_0 I_d \right),
\end{aligned}
\end{equation}
where $m$ denotes the Lebesgue measure in $\mathbb{R}^1$. Clearly, $\left\{tx+(1-t)y : t\in(0,1) \right\}$ is a segment connecting the two point $x,y$ in $\mathbb{R}^d$. Then by definition of the ball $B(0,R_0) := \left\{x\in \mathbb{R}^d : |x| < R_0\right\}$, the longest segment contained in $B(0,R_0)$ is of length $2R_0$. Therefore,
\begin{equation}\label{eq:1:7}
    m(A_1) \leq \frac{2R_0}{r},\quad m(A_2) = 1 - m(A_1) \geq 1 - \frac{2R_0}{r}.
\end{equation}
Combining \eqref{eq:1:5}, \eqref{eq:1:6} and \eqref{eq:1:7}, for all $x,y \in \mathbb{R}^d$ satisfying $|x-y| = r$,
\begin{equation}
    \frac{(x-y) \cdot (\nabla U(x) - \nabla U(y))}{|x-y|^2} \geq -\frac{2R_0}{r}\,K + \left(1 - \frac{2R_0}{r} \right)\, \kappa_0 = \kappa_0 - \frac{2R_0}{r} (K + \kappa_0).
\end{equation}
Choosing $R:= \max( 4R_0(K+\kappa_0)/\kappa_0, 2)$, the conclusion \eqref{eq:lmmass} then holds with $\kappa := \kappa_0 / 2$.
\end{proof}

Next, we prove the result of $p$-th moment control for SGLD.

\begin{proof}[Proof of Lemma \ref{lmm:sgldmoment}:] 
We first control the moments of $\bar{X}_t$, namely $\E|\bar{X}_t|^p$ on $[0, T]$, and then prove the moment control of $\sup_{t\le T}|\bar{X}_t|$.

We take $p\ge 2$ first. By It\^o's formula, for $t \in [T_k,T_{k+1})$, we have
\begin{multline}
    d|\bar{X}_t|^p = -p|\bar{X}_t|^{p-2} \bar{X}_t \cdot \nabla U^{\xi_k}(\bar{X}_{T_k}) dt + \beta^{-1}p |\bar{X}_t|^{p-2} \left( I_d + (p-2)\frac{\bar{X}_t^{\otimes 2}}{|\bar{X}_t|^2}\right) : I_d\, dt \\
      -p|\bar{X}_t|^{p-2} \bar{X}_t \cdot \sqrt{2\beta^{-1}} dW.
\end{multline}
Note that $\left( I_d + (p-2)\frac{\bar{X}_t^{\otimes 2}}{|\bar{X}_t|^2}\right) : I_d=p+d-2$. 
This implies that
\begin{equation}\label{eq:2__9}
    \frac{d}{dt}\mathbb{E}|\bar{X}_t|^p = \mathbb{E}\left[-p|\bar{X}_t|^{p-2} \bar{X}_t \cdot \nabla U^{\xi_k}(\bar{X}_{T_k})\right] +\beta^{-1}p(p+d-2) \mathbb{E}|\bar{X}_t|^{p-2} .
\end{equation}

By the Lipschitz condition in Assumption \ref{ass}, we can directly obtain
\[
\frac{d}{dt}\mathbb{E}|\bar{X}_t|^p
\le (p-1)\E|\bar{X}_t|^p+C(p,K)(1+\E|\bar{X}_{T_k}|^p)+\beta^{-1}p(p+d-2) \mathbb{E}|\bar{X}_t|^{p-2}.
\]
This easily yields 
\begin{gather}\label{eq:momentctrl1}
\sup_{t\le T}\mathbb{E}|\bar{X}_t|^p<\infty,
\end{gather}
where the upper bound depends on $p, T, d$ but is independent of $\xi_k$.

Next, we prove
\begin{equation}
    \sup_{0\leq t\leq T} \mathbb{E}\left[\sup_{0\leq s\leq t}\left|\bar{X}_s\right|^p\right] < +\infty.
\end{equation}
Note that $|\bar{X}_t|^p = |\bar{X}_{0}|^p + M_t + A_t$, where
\begin{equation}
    M_t := \int_{0}^t \sqrt{2\beta^{-1}} p |\bar{X}_s|^{p-2} \bar{X}_s \cdot dW_s, \quad \forall t \geq 0,
\end{equation}
and
\begin{equation}
\begin{split}
&    A_t := -\int_{0}^t p |\bar{X}_s|^{p-2} \bar{X}_s \cdot b_s\, ds + \int_{0}^t \beta^{-1} p(p+d-2) |\bar{X}_s|^{p-2}  ds,\\
&    b_t := \nabla U^{\xi_k} (\bar{X}_{T_k}), \quad \forall t \in [T_k,T_{k+1}).
\end{split}
\end{equation}
Then,
\begin{equation}\label{eq:4_32}
    \mathbb{E}\left[\sup_{0 \leq s \leq t} |\bar{X}_s|^p \right] \leq \mathbb{E}|\bar{X}_{0}|^p + \mathbb{E}\left[\sup_{0 \leq s \leq t} M_s\right] + \mathbb{E}\left[\sup_{0 \leq s \leq t} A_s\right].
\end{equation}
Clearly $M_t$ is a martingale. By  BDG inequality \cite{mao2007stochastic}, one has
\begin{equation*}
    \mathbb{E}\left[\sup_{0 \leq s \leq t} M_s\right] \leq 4\sqrt{2}\mathbb{E}\left[\left\langle M\right\rangle_t^\frac{1}{2}\right] = 8\sqrt{\beta^{-1}}p \mathbb{E}\left[\left(\int_{0}^t |\bar{X}_s|^{2p-2} ds\right)^{\frac{1}{2}}\right].
\end{equation*}
 Then using Jensen's inequality and \eqref{eq:momentctrl1}, one has
\begin{equation}\label{eq:supMt}
    \mathbb{E}\left[\sup_{0 \leq s \leq t} M_s\right] \leq   8\sqrt{\beta^{-1}}p\left(\int_{0}^t \mathbb{E}|\bar{X}_s|^{2p-2} ds\right)^{\frac{1}{2}} \leq C_2.
\end{equation}

For the term $\mathbb{E}\left[\sup_{0 \leq s \leq t} A_s\right]$, using the Lipshitz condition in Assumption \ref{ass}, we first observe that for $t\in [0,T]$,
\begin{equation*}
    A_t \leq \int_0^t p|\bar{X}_s|^{p-1} \left(K\sup_{0 \leq u \leq s} |\bar{X}_u| + b_0 \right) ds + C_3\int_0^t |\bar{X}_s|^{p-2} ds,
\end{equation*}
where $b_0$, $C_3$ are time-independent positive constants. Applying Young's equality, we have
\begin{equation}\label{eq:supAt}
    \mathbb{E}\left[\sup_{0 \leq s \leq t} A_s\right] \leq (pK + 1) \int_0^t \mathbb{E}\left[\sup_{0 \leq u \leq s}|\bar{X}_u|^p\right] ds + C_4.
\end{equation}

Combining \eqref{eq:4_32}, \eqref{eq:supMt} and \eqref{eq:supAt} together, one has that
\begin{equation}
    \mathbb{E}\left[\sup_{0 \leq s \leq t}|\bar{X}_s|^p\right] \leq C_5 + (pK + 1) \int_0^t \mathbb{E}\left[\sup_{0 \leq u \leq s}|\bar{X}_u|^p\right] ds, \quad \forall t \in [0,T].
\end{equation}
Hence, by Gr\"onwall's inequality, for all $t \in [0,T]$, we have
\begin{equation}
    \mathbb{E}\left[\sup_{0 \leq s \leq t} |\bar{X}_s|^p\right] \leq C_5 e^{(pK+1)T}.
\end{equation}
The bound for $p\in [1, 2]$ then follows easily by H\"older's inequality.

Next, we aim to establish the uniform moment control of $\bar{X}_t$ for $\eta_k$ being sufficiently small. Starting with \eqref{eq:2__9}, the first term on the right hand side may be written as
\begin{multline}\label{eq:2__11}
    \mathbb{E}\left[-p|\bar{X}_t|^{p-2} \bar{X}_t \cdot \nabla U^{\xi_k}(\bar{X}_{T_k})\right] 
    = \mathbb{E}\left[-p|\bar{X}_{T_k}|^{p-2} \bar{X}_{T_k} \cdot \nabla U^{\xi_k}(\bar{X}_{T_k})\right]\\
     + p \mathbb{E}\left[\int_0^1 \nabla h\left(\lambda \bar{X}_t + (1-\lambda) \bar{X}_{T_k}\right)d\lambda \cdot \left(\bar{X}_t - \bar{X}_{T_k}\right) \right] := p(K_1 + K_2),
\end{multline}
where the function $h$ is defined by $h(x) := -|x|^{p-2} x \cdot \nabla U^{\xi_k}(\bar{X}_{T_k})$, and has a well-defined gradient $\nabla h(x) = -\nabla U^{\xi_k}(\bar{X}_{T_k}) \cdot |x|^{p-2} \left(I_d + (p-2) \frac{x^{\otimes 2}}{|x|^2} \right)$. Since $\bar{X}_t-\bar{X}_{T_k}=-(t-T_k)\nabla U^{\xi_k}(\bar{X}_{T_k})+\sqrt{2\beta^{-1}}(W_t-W_{T_k})$, 
\begin{multline*}
K_2 \le  (p-1)\eta_k \int_0^1 \E|\lambda \bar{X}_t + (1-\lambda) \bar{X}_{T_k}|^{p-2}|\nabla U^{\xi_k}(\bar{X}_{T_k})|^2d\lambda\\
+(p-1)\sqrt{2\beta^{-1}}\int_0^1 \E|\lambda \bar{X}_t + (1-\lambda) \bar{X}_{T_k}|^{p-2}|\nabla U^{\xi_k}(\bar{X}_{T_k})|\left|\int_{T_k}^t dW\right|d\lambda.
\end{multline*}
Note that $p-2\ge 0$, 
$|\lambda \bar{X}_t + (1-\lambda) \bar{X}_{T_k}|^{p-2} \le \max(|\bar{X}_t|^{p-2}, |\bar{X}_{T_k}|^{p-2})
\le (|\bar{X}_t|^{p-2}+|\bar{X}_{T_k}|^{p-2})$.
Then, one has
\begin{multline*}
K_2\le (p-1)\eta_k(1+\delta_1)K^2
\E(|\bar{X}_t|^{p-2}|\bar{X}_{T_k}|^2+|\bar{X}_{T_k}|^p) 
+C_{\delta_1}b_0^2(p-1)\eta_k\E(|\bar{X}_t|^{p-2}+|\bar{X}_{T_k}|^{p-2})\\
+\sqrt{\beta^{-1}\eta_k}C(p,d)
\left[(\E|\bar{X}_t|^p)^{(p-1)/p}
+(\E|\bar{X}_{T_k}|^p)^{(p-1)/p}+1\right].
\end{multline*}

 For the term $K_1$, by Assumption \ref{ass} and Lemma \ref{lmm:ass}, 
 \[
 -x\cdot \nabla U(x)\le -\kappa|x|^2
 +b_0|x|+C(R).
 \]
Using the consistency of the random batch $\xi$, we have
\begin{equation}\label{eq:K_1}
    \begin{aligned}
        K_1 & =\mathbb{E}\left[-|\bar{X}_{T_k}|^{p-2} \bar{X}_{T_k} \cdot \nabla U(\bar{X}_{T_k})\right] =\mathbb{E}\left[-|\bar{X}_{T_k}|^{p-2} \bar{X}_{T_k} \cdot \nabla U(\bar{X}_{T_k}) \right]
        \\
        &\leq  -\kappa\mathbb{E}|\bar{X}_{T_k}|^p + b_0 \mathbb{E}|\bar{X}_{T_k}|^{p-1}
        +C(R)\mathbb{E}|\bar{X}_{T_k}|^{p-2}.
    \end{aligned}
\end{equation}
Let $\epsilon_k=(p-1)\eta_k (1+\delta_1)K^2$. Then by Young's inequality, we conclude that
\begin{equation}\label{eq:2__14}
   p( K_1 + K_2) \leq -\left(\kappa-\epsilon_k(1+\frac{2}{p})-\delta_2\right) \mathbb{E}|\bar{X}_{T_k}|^p+\left(\epsilon_k(1-\frac{2}{p})+\delta_2\right)\mathbb{E}|\bar{X}_{t}|^p + C.
\end{equation}
Letting $u(t):=\E|\bar{X}_t|^p$, one then has for $t\in [T_k, T_{k+1}]$ that
\begin{gather}
\dot{u}(t)\le -\left(\kappa-\epsilon_k(1+\frac{2}{p})-\delta_2\right) u(T_k)+\left(\epsilon_k(1-\frac{2}{p})+\delta_2\right)u(t) + C(\delta_1,\delta_2, p, d).
\end{gather}
For $0<\lambda_1<\lambda_2$ and $v\ge 0$ satisfying
\[
\dot{v}\le \lambda_1 u(t)-\lambda_2 u(T_k)+C,
\]
one may obtain by Gr\"onwall's inequality that
\[
v(T_{k+1})\le \left(e^{\lambda_1\eta_k}-\frac{1}{\lambda_1}(e^{\lambda_1 \eta_k}-1)\lambda_2\right)v(T_k)+C\frac{1}{\lambda_1}(e^{\lambda_1 \eta_k}-1).
\]
However, since 
\[
e^{\lambda_1\eta_k}-\frac{1}{\lambda_1}(e^{\lambda_1 \eta_k}-1)\lambda_2
=(1-\frac{\lambda_2}{\lambda_1})e^{\lambda_1\eta_k}+\frac{\lambda_2}{\lambda_1}\le 1+(\lambda_1-\lambda_2)\eta_k,
\]
one then has
\[
v(T_{k+1})\le [1+(\lambda_1-\lambda_2)]v(T_k)
C\frac{1}{\lambda_1}(e^{\lambda_1 \eta_k}-1).
\]
We apply this elementary derivation for
$\lambda_1=\epsilon_k(1-2/p)+\delta_2$
and $\lambda_2=\kappa-\epsilon_k(1+2/p)-\delta_2$, then obtain
\[
u(T_{k+1})\le [1-(\kappa-2\epsilon_k-2\delta_2)\eta_k ]u(T_k)+C(p, d, \beta, \delta_1, \delta_2, \eta_k).
\]
Since we can choose $\delta_1$ and $\delta_2$ small, by the condition $\eta_k \le \kappa/(2(p-1)K^2)-\delta$ given, $\kappa-2\epsilon_k-2\delta_2$ is bounded below by a positive number and $C(p, d, \beta, \delta_1, \delta_2, \eta_k)$ has a uniform upper bound in $k$. Moreover, since $\kappa \le K$, $(\kappa-2\epsilon_k-2\delta_2)\eta_k < 1$. The claim then follows. 
\end{proof}

\section{Details for construction of reflection coupling and Lyapunov function}\label{app:detail}
Here we present more details for the principal method employed in this study - reflection coupling equipped with a specific Lyapunov function $f(\cdot)$, as described in the introduction.

Consider the two time marginal distributions $\rho^{(1)}_t$, $\rho^{(2)}_t$ of some SDE (in our result, it is \eqref{eq:sgld}), starting from the initial distributions $\rho^{(1)}_0$, $\rho_0^{(2)}$, respectively. As has been discussed in the introduction, here we aim to prove the contraction property: 
$$
W_f(\rho^{(1)}_t,\rho^{(2)}_t) \lesssim e^{-ct}W_f(\rho^{(1)}_0,\rho^{(2)}_0).
$$ 
Here, $f(\cdot)$ is some suitable Lyapunov function and $W_f(\cdot,\cdot)$ is the Kantorovich-Rubinstein distance associated with the cost function $f(\cdot)$. The reflection coupling method begins with choosing the pair of initial points $(X_0, Y_0)$ such that $\mathbb{E}f(|X_0-Y_0|) = W_f(\rho^{(1)}_0,\rho^{(2)}_0)$.
Then we choose a realization $\bar{X}_t$ of SGLD \eqref{eq:sgld} such that the law of $X_t$ is $\rho^{(1)}_t$ and the law of $X_0$ is $\rho^{(1)}_0$. The key step in the reflection coupling method is that we construct a companion process $Y_t$ with $Y_0$ coupled above with $X_0$ and satisfies: (i) $Y_t$ shares the same Brownian motion with $\bar{X}_t$, and has an additional reflection term in its diffusion part, and $Y_t$ also shares the same random batch $\xi_k$ at each $T_k$ in our SGLD setting \eqref{eq:sgld}; (ii) $Y_t$ is also a realization of the same SDE for $X_t$ and the law of $Y_t$ is $\rho^{(2)}_t$. Then the contraction property mentioned above is reduced to estimation of the negative Lyapunov exponent for the paired dynamics ($X_t$, $Y_t$). Namely, we aim to show that
\begin{equation*}
    \mathbb{E}f(|X_t - Y_t|) \leq C e^{-Ct}\mathbb{E}f(|X_0 - Y_0|)
\end{equation*}

In the followings, we will first introduce the necessity of using the technique of reflection coupling, and then introduce the motivation of the construction of the reflection coupling and the associated Lyapunov function. Note that the geometric ergodicity arises from the strong convexity of the potential $U(\cdot)$ outside some compact set. In fact, by strong monotonicity property in Lemma \ref{lmm:ass}
\begin{equation*}
    (x - y) \cdot (\nabla U(x) - \nabla U(y)) \geq \kappa |x - y|^2, 
\end{equation*}
any such pair ($X_t$, $Y_t$) would attract each other if they are sufficiently far away. Take the following numerical scheme for SDE as a simple illustration:
\begin{equation}\label{eq:numericalsde}
    X^{n+1} = X^n - \eta \nabla U (X^n) + \sqrt{\eta}\zeta, \quad \zeta \sim \mathcal{N}(0,1).
\end{equation}
In the settings of this paper, on one hand, as mentioned above, the strong convexity outside some compact set of the potential $U(\cdot)$ in the drift implies that any paired iteration ($X^n$, $Y^n$) associated with \eqref{eq:numericalsde} would attract each other if they are far away.
On the other hand, the external force is weak inside the compact set since in this area the potential $U(\cdot)$ does not have strong convexity. In this case, the diffusion term would dominates the drift, since $c_1\sqrt{\eta} \leq c_2\eta$, where $c_1$, $c_1$ are of $O(1)$. Therefore, at first glance, one cannot directly prove the contraction when the diffusion overshadows the drift. Nonetheless, the application of reflection coupling \cite{eberle2011reflection} offers a resolution by facilitating the closer convergence of two particles $X_t$, $Y_t$ even within the compact set. Take the following overdamped Langevin diffusion for example:
\begin{equation*}
    dX_t = b(X_t)dt + dW, \quad X|_{t=0} = X_0.
\end{equation*}
The reflection coupling method for the overdamped Langevin diffusion considers another slave copy of $X_t$, which shares the same Brownian motion but has a reflection term in the diffusion part:
\begin{equation*}
    dY_t = b(Y_t)dt + \left(I_d - 2\frac{\left(X_t - Y_t \right)^{\otimes 2}}{|X_t - Y_t|^2} \right) \cdot dW, \quad Y|_{t=0} = Y_0.
\end{equation*}
It can be shown that the diffusion with a reflection is still a Brownian motion (see \cite{eberle2011reflection} or Lemma \ref{lmm:verifyBM}), so $Y_t$ is also a realization of the overdamped Langevin diffusion. With the reflection matrix $\left(I_d - 2\frac{\left(X_t - Y_t \right)^{\otimes 2}}{|X_t - Y_t|^2} \right)$, the two particles $X_t$, $Y_t$ would eventually move towards each other when the diffusion dominates the drift. In fact, from the reflection operator $\left(I_d - 2\frac{\left(X_t - Y_t \right)^{\otimes 2}}{|X_t - Y_t|^2} \right)$, the Brownian motion can be either approaching or depart from each other. However, the restored force would prevent $X_t$, $Y_t$ from going too far away from each other. With this intuitive picture of how the reflected couple $(X_t, Y_t)$ moves, it is then left to prove the contraction via practical calculation, namely, one needs to find some Lyapunov function $f(\cdot)$ satisfying
\begin{enumerate}
    \item $C_1 r \leq f(r) \leq C_2 r$ for all $r$. 2. $\mathbb{E} f(|X_t - Y_t|)$ decays exponentially in time.
\end{enumerate}
Below we discuss a bit on our motivation for how to fund such Lyapunov function. One can see from It\^o's formula that
\begin{equation}\label{eq:conti_demo}
    \frac{d}{dt} \mathbb{E} \left[ f(|X_t - Y_t|) \right] = \mathbb{E}\left[f''(|X_t - Y_t|) + f'(|X_t - Y_t|) \frac{X_t - Y_t}{|X_t - Y_t|} \cdot \left(b(X_t) - b(Y_t) \right) \right].
\end{equation}
Since the goal is to obtain an estimate of the form
\begin{equation*}
    \frac{d}{dt} \mathbb{E} \left[ f(|X_t - Y_t|) \right] \lesssim - \mathbb{E} \left[ f(|X_t - Y_t|) \right],
\end{equation*}
one naturally requires the following conditions when constructing such $f$: (1) $C_1r \leq f(r) \leq C_2 r$; (2) $|f'(r)| \leq L$; (3) $f''(r) \leq -C_3r$ for all $r<R_1$, where $R_1$ is some positive constant larger than $R$. If these conditions are satisfied, then one can see from \eqref{eq:conti_demo} that
\begin{equation*}
\begin{aligned}
    \frac{d}{dt} \mathbb{E} \left[ f(|X_t - Y_t|) \textbf{1}_{\{|X_t - Y_t| < R \} }\right] &\leq \mathbb{E}\left[\left(-C_3|X_t - Y_t| + L \| b' \|_{\infty} |X_t - Y_t| \right)\textbf{1}_{\{|X_t - Y_t| < R \}}\right]\\
    &\lesssim -\mathbb{E} \left[ f(|X_t - Y_t|) \textbf{1}_{\{|X_t - Y_t| < R \} }\right], 
\end{aligned}
\end{equation*}
provided that $C_3$ is relatively large. This then motivates one to seek a concave increasing Lyapunov function $f$ of the form 
\begin{equation*}
    f(r) := \int_0^r e^{-c_f (s \wedge R_1)} ds, \quad r \geq 0.
\end{equation*}
for some positive $c_f$, $R_1$ to be determined (in our result for SGLD, we choose $R_1 = 2R$ and the required condition for $c_f$ is stated in \eqref{eq:condcf}).  Then one can obtain the contraction property for this reflection coupled continuous dynamics $(X_t, Y_t)$.

\bibliographystyle{plain}
\bibliography{main}

\end{document}